\documentclass[12pt,a4paper,oneside]{article}
\usepackage{amsmath,amstext,amssymb,amsthm}
\newtheorem{thm}{Theorem}[section]
\newtheorem{lem}[thm]{Lemma}
\newtheorem{prop}[thm]{Proposition}
\newtheorem{cor}[thm]{Corollary}
\newtheorem{rem}[thm]{Remark}

\newcommand{\pref}[1]{(\ref{#1})}
\newcommand{\Rd}{\mathbb{R}^d}

\renewcommand{\H}{\mathbb{H}}

\def\kl{\left\{}
\def\kr{\right\}}

\title{Two-sided optimal bounds for Green functions of half-spaces  for relativistic $\alpha$-stable process}
\author{ T. Grzywny, M. Ryznar \\
 Institute of Mathematics and Computer Sciences,\\
  Wroc\l{}aw University of Technology, Poland}

\date{}
\begin{document}
\maketitle

\begin{abstract}
The purpose of this paper is to find optimal estimates for the
Green function of a half-space of {\it the relativistic
$\alpha$-stable process} with parameter $m$
 on $\Rd$ space. This process has an infinitesimal generator of the form
 $mI-(m^{2/\alpha}I-\Delta)^{\alpha/2},$ where
  $0<\alpha<2$, $m>0$, and reduces to the isotropic  $\alpha$-stable process for $m=0$ .
 Its potential theory for open bounded sets has been well developed throughout the recent years
 however almost nothing was known about the behaviour of the process on unbounded sets. The present
  paper is intended to fill this gap and we provide two-sided sharp estimates for the Green
   function for a half-space. As a byproduct we obtain some improvements
   of the estimates known for bounded sets.

   Our approach  combines the recent results obtained in \cite{ByczRyzMal}, where an
   explicit integral formula for the $m$-resolvent of a half-space was found, with estimates of
   the transition densities for the killed process on exiting a half-space.
 The main result states that the Green function is comparable with the  Green function for the
 Brownian motion if the points are away from the boundary of a half-space  and their distance  is greater than one. On the
  other hand for the remaining points  the Green function is somehow
  related  the  Green function for the isotropic  $\alpha$-stable process. For example, for $d\ge3$, it is comparable
  with the  Green function for the isotropic  $\alpha$-stable process, provided that the points are close enough.
\end{abstract}

\textbf{Keywords:} {stable relativistic process, Green function,
first exit time from a ball, tail function}\\

\textbf{Mathematics Subject Classifications (2000):} {60J45}

\section{Introduction}
In the paper we deal with some aspects of the potential theory of
the $\alpha$-{\it stable relativistic process}. That is a L{\'e}vy
process on $\Rd$ with a generator of the form
$$
 H^{m}_{\alpha}=mI-(m^{2/\alpha} I-\Delta)^{\alpha/2},\quad\ 0<\alpha <2, \ m>0.
 $$
For $m=0$ the operator above reduces to the generator of the
$\alpha$-{\it stable rotation invariant
 (isotropic) L{\'e}vy process} which potential theory was intensively studied in the literature.

 For $\alpha=1$ the operator
  $$
 H^{m}_{1}=mI-(m^2 I-\Delta)^{1/2}
 $$
 plays a very important role  in relativistic quantum mechanics since it corresponds to the kinetic energy of a relativistic particle with mass $m$.
  Generators of this kind  were investigated for example by E. Lieb \cite{L}
   in connection
  with the problem of stability of relativistic matter.
  An interested reader will find references on this subject e.g. in a recent paper
  \cite{KS}.

  Another reason that the operator $H^{m}_{\alpha}$ is an interesting object of study is its role in the theory of
the so-called {\it interpolation spaces of Bessel potentials} and
its application in harmonic analysis and partial differential
equations (see, e.g. \cite{S} and \cite{Ho}). This theory is based
on {\it Bessel potentials} defined as $J_{\alpha}=
(I-\Delta)^{-\alpha/2}$. As Stein pointed out in his monograph
\cite{S}, the Bessel potentials exhibit the same {\it local}
 behaviour (as $|x| \to 0$) as the Riesz potentials   but the {\it global} one (as $|x| \to \infty$)
 of  $J_{\alpha}$ is much more regular. In terms of the relativistic process the potential  $J_{\alpha}$
 is so-called $1$-resolvent kernel of the semigroup generated by $ H^{1}_{\alpha}$.

 In the paper we consider the process killed on exiting the  half-space $\H=\{x\in \Rd: x_d>0 \}$ and examine
 the behaviour of its  Green function $G_\H(x,y)$. Contrary to the stable case a closed formula for that Green
 function is not know and seems to be a very challenging target.  Recently in \cite{ByczRyzMal} an integral
 formula in terms of the Macdonald functions was found for $G^m_\H(x,y)$ - the $m$-resolvent kernel for $\H$.
 As proved in  \cite{ByczRyzMal}, for $d\ge 3$, the behaviour of the Green function is equivalent to  the behaviour
 the $m$-resolvent if $|x-y| \to 0$. Our main result establishes optimal bounds for the Green function of  $\H$.
 To our best knowledge it is the first result of that type when optimal estimates for unbounded set
 (different than the whole  $\Rd$) are derived.

 At this point let us mention that the potential theory for bounded sets has been well developed
  during recent years (see \cite{CS}, \cite{Ry}, \cite{KL}, \cite{GRy}). Under various assumptions of the regularity of
   a bounded open set $D$ it was shown that the Green function of $D$ was comparable with its stable
   counterpart. This comparison allowed to prove the relativistic potential theory shares most of
   the properties of the stable one if bounded sets are considered.  Comparing   the potential kernel
    for the stable process with the potential kernel for the relativistic process (see \cite{RSV}) we can
     conclude that such a comparison of Green functions is not generally possible for unbounded sets.
 Since the relativistic potential kernel (for $d\ge 3$) is asymptotically equivalent (if $|x-y|$ is large)
 to that of the Brownian
 motion it may suggest that the Green function of $\H$, at least for some part of  the range of $x,y$,
 is comparable with  the  Green function of    $\H$ for the  Brownian motion. Our  main result confirms
 that suggestion  and we prove the comparability  for points $x,y$ being away from the boundary and
 with $|x-y|\ge 1$. For other points our bound is also optimal.

 We also thoroughly examine the one-dimensional case and provide
 optimal estimates for the Green functions for bounded intervals taking into account their length.
 While for intervals of moderate length (say smaller than $1$) we can use the well known results about
 comparability  of stable and relativistic Green functions, for large intervals we relay on the estimates
 for half-lines obtained in this paper. Again we show that the Green functions for large intervals are
 comparable to  the Brownian  Green functions for most of the range.

 The organization of the paper is as follows.
 In Section 2 we collect all  definitions and preliminary results needed for the rest of the paper.
 The next section is basic for the paper. Here we prove the estimates for the Green function of
  $(0,\infty)$. Then in Section 4 we apply them to prove the optimal bounds for the tail function of
  the exit time from $(0,\infty)$ and some other properties of the exit times. These estimates
  will have a crucial role in examining multidimensional case which was accomplished
  in  Section 5. We conclude the paper with exploring in the last section the one-dimensional case
  with regard to optimal estimates for  bounded intervals.

\section{Preliminaries}
\setcounter{equation}{0}

    Throughout the paper  by $c, C, C_1\,\dots$ we  denote
       nonnegative constants which may depend on other constant parameters only.
        The value of $c$ or $C, C_1\,\ldots$ may change from line to line in a chain
         of estimates.

      The notion $p(u)\approx q(u),\ u \in A$ means that the ratio
      $p(u)/ q(u),\ u \in A$ is bounded from  below and above
      by positive constants which may depend on other constant parameters only but does not depend on the set $A$.

We present in this section some basic material regarding the
$\alpha$-stable relativistic process. For more  detailed
information, see \cite{Ry} and \cite{C}. For questions regarding
Markov and strong Markov property, semigroup properties,
Schr\"{o}dinger operators and basic potential theory, the reader
is referred to \cite{ChZ} and \cite{BG}.

We first introduce an appropriate class of subordinating
processes.
 Let $\theta_{\alpha}(t,u)$, $u, t>0$, denote the density function
of the strictly $\alpha/2$-stable positive standard subordinator,
$0<\alpha<2$, with the Laplace transform $e^{-t
\lambda^{\alpha/2}}$.

Now for $m>0$ we define another subordinating process
$T_{\alpha}(t,m)$ modifying the corresponding probability density
function in the following way:
 $$
 \theta_{\alpha}(t,u,m)= e^{mt}\,\theta_{\alpha}(t,u)\,e^{-m^{2/\alpha}u},  \quad u>0\,.
 $$
 We derive the Laplace transform of $T_{\alpha}(t,m)$ as follows:
 \begin{equation} \label{subord2}
 E^{0} e^{-\lambda T_{\alpha}(t,m)} =e^{mt}\, e^{-t (\lambda+m^{2/\alpha})^{\alpha/2}}, \quad \lambda\ge -m^{2/\alpha}.
 \end{equation}
 Let $B_t$ be the symmetric Brownian motion in $\Rd$ with the characteristic
  function of the form
 \begin{equation} \label{brownian}
 E^{0}e^{i\xi \cdot B_t} = e^{-t|\xi|^2}\,.
 \end{equation}
 Assume that the processes $T_{\alpha}(t,m)$ and $B_t$ are stochastically independent.
 Then the process $X_t^{\alpha,m}= B_{T_{\alpha}(t,m)}$ is called the
  $\alpha$-stable relativistic process
 (with parameter $m$).
  In the sequel we use the generic notation $X^m_t$ instead of
   $X_t^{\alpha,m}$.
  If $m=1$ we write
  $T_{\alpha}(t)$ instead of $T_{\alpha}(t,m)$ and $X_t$ instead of $X_t^1$. From   (\ref{subord2}) and (\ref{brownian}) it is clear that
  the characteristic
  function of $X^m_t$ is  of the form
 $$ 
 E^{0}e^{i\xi \cdot X^m_t } = e^{mt}e^{-t(|\xi|^2+m^{2/\alpha})^{\alpha/2}}\,.
 $$
Obviously in the case $m=0$ the corresponding process is the
standard (rotationally invariant or isotropic) $\alpha$-stable
process.
  $X^m_t$ is a L\'evy process (i.e. homogeneous, with independent
 increments). We always assume that sample paths of the process $X^m_t$ are right-continuous
 and have left-hand limits ("cadlag"). Then $X^m_t$ is Markov and has the strong
 Markov property under the so-called standard filtration.

 From the form of the Fourier transform we have the following scaling property:
 \begin{equation}\label{scaling}
 {p}^m_{t}(x)=m^{d/\alpha}{p}^1_{mt}(m^{1/\alpha}x).
   \end{equation}
 In terms of one-dimensional distributions of the relativistic process (starting from
 the point $0$) we obtain
 $$
 X^m_t \sim m^{-1/\alpha} X_{mt}\,,
 $$
 where $X_t$ denotes the relativistic $\alpha$-stable process with parameter $m=1$
 and "$\sim$" denotes equality of distributions.
 Because of this scaling property, we usually restrict our attention to the case
  when $m=1$, if not specified otherwise.  When $m=1$ we omit the superscript "$1$", i.e. we write $p_{t}(x)$
  instead of $p_t^1(x)$, etc.

 Various potential-theoretic objects in the theory of the process $X_t$ are expressed
 in terms of modified Bessel functions $K_{\nu}$ of the second kind, called also Macdonald
 functions. For convenience of the reader
 we collect here basic information about these functions.

  $K_\nu,\ \nu\in \mathbb{R} $, the modified Bessel function of the second kind
   with index $\nu$,  is given by the
 following formula:
 $$ K_\nu(r)=  2^{-1-\nu}r^\nu \int_0^\infty e^{-u}e^{- \frac{r^2}{4u}}u^{-1-\nu}du\,,
 \quad r>0.$$
  For properties of  $K_\nu$ we refer the reader to \cite{E1}. In the sequel we will
  use the asymptotic behaviour of $K_\nu$:

 \begin{eqnarray}
  K_\nu(r)& \cong& \frac{\Gamma(\nu)}{2} \left(\frac{r}{2}\right)^{-\nu},\quad  r\to 0^+, \ \nu>0,
  \label{asympt0}\\
    K_0(r)&\cong& -\log r,  \quad r\to 0^+, \label{asympt00}\\
    K_\nu(r)&\cong& \frac {\sqrt{\pi}} {\sqrt{2 r}}\,e^{-r},  \quad r\to \infty, \label{asympt_infty}
 \end{eqnarray}
 where  $g(r) \cong f(r) $ denotes that the ratio of $g$ and $f$ tends to $1$.
  For  $\nu <0 $ we have $K_\nu(r)=K_{-\nu}(r)$, which determines the asymptotic
  behaviour for negative indices.


  The  $\alpha$-stable relativistic  density (with parameter $m=1$)
   can now be computed in the following way:
  \begin{equation} \label{reldensity0}
 p_t(x)=\int_0^\infty e^{t}\,\theta_{\alpha}(t,u)\,e^{-u}\, g_u(x) du,
 \end{equation}
 where $g_u(x)=\frac1{(4\pi u)^{d/2}}e^{- \frac{|x|^2}{4u}}$ is  the Brownian semigroup, defined by \pref{brownian}.

We also recall the form of the density function $\nu(x)$ of the
L{\'e}vy measure of the
 relativistic $\alpha$-stable process (see \cite{Ry}):
  \begin{eqnarray}
  \nu(x)&=& \frac \alpha{2\Gamma(1- \frac{\alpha }{2})}\int_0^\infty e^{-u}\, g_u(x)\, u^{-1-\alpha/2} du\label{levymeasure0}\\
  &=&
  \frac{\alpha 2^{\frac{\alpha -d}{ 2}}}{ \pi^{d/2} \Gamma(1- \frac{\alpha }{2})}
   |x|^{ -\frac{d+\alpha }{ 2}}
  K_{ \frac{d+\alpha }{ 2}}(|x|)\,.\label{levymeasure}
  \end{eqnarray}

In the case $0<\alpha <2$ we have the following useful estimates
(see \cite{Ry} for the proof of the first lemma):

  \begin{lem} \label{transden_0}
  There exists a constant $c=c(\alpha,d)$ such that
  \begin{equation} \label{transupper}
  \max_{x \in {\Rd}} p_t(x) \le c (t^{-d/2} + t^{-d/{\alpha}}) \,.
  \end{equation}
  \end{lem}

\begin{lem} \label{transden2}
For any $t>0$ and $x\in \Rd$ we have
$$p_t(x)\le c(d,\alpha) \left(g_{t}(x/\sqrt{2})+t\nu(x/\sqrt{2})\right)$$
and
$$p_t(x)\le \frac{c(d)}{|x|^d}.$$
\end{lem}

\begin{proof}
Notice that for $u,t>0$,
$$\theta_{\alpha}(t,u)=t^{-2/\alpha}\theta_{\alpha}(1,t^{-2/\alpha}u) \quad \textrm{and}\quad \theta_{\alpha}(1,u)\le c u^{-1-\alpha/2}.$$
Hence
\begin{equation}\label{thetagora}\theta_{\alpha}(t,u)\le c t u^{-1-\alpha/2},\qquad t,u>0.\end{equation}
Using \pref{reldensity0} we obtain for $t\ge 1$,
\begin{eqnarray*}
p_t(x)&\le&
e^{-\frac{|x|^{2}}{8t}}(4\pi)^{-\frac{d}{2}}e^{t}\int^{2
t}_{0}\theta_{\alpha}(t,u)e^{-u}u^{-d/2}du\\& & +\; c t
\int^{\infty}_{2t}g_u(x)e^{-\frac{u}{2}}u^{-1-\alpha/2}du\\
&\le& e^{-\frac{|x|^{2}}{8t}}(4\pi)^{-\frac{d}{2}}e^{t}
\int^{\infty}_{0}\theta_{\alpha}(t,u)e^{-u}u^{-d/2}du\\ & & +\; c
t
\int^{\infty}_{0}g_u(x)e^{-\frac{u}{2}}u^{-1-\alpha/2}du\\
&=& e^{-\frac{|x|^{2}}{8t}}p_t(0)+c t\nu(x/\sqrt{2}),
 \end{eqnarray*}
 where we used \pref{levymeasure0} in the last line. Moreover by Lemma \ref{transden_0}
 we can estimate
 $$e^{-\frac{|x|^{2}}{8t}}p_t(0)\le c g_{t}(x/\sqrt{2}),\quad  t\ge 1.$$
 This completes the proof of the first estimate for $t\ge 1$.

Next, for $t\le 1$, applying (\ref{reldensity0}),
(\ref{thetagora}) and (\ref{levymeasure0}) we arrive at
$$
p_t(x)\le c t \int^\infty_0 g_u(x)e^{-u}u^{-1-\alpha/2}du=c t
\nu(x)\le c t \nu(x/\sqrt{2}),
$$
which complete the proof the first inequality.

The second bound is true for the transition density of any
subordinated Brownian motion. Indeed let us observe that for any
$t>0$ and $x\in\Rd,$
$$g_t(x)\le \left(\frac{d}{2\pi}\right)^{d/2}e^{-\frac{d}{2}}|x|^{-d}.$$
Hence by subordination $$p_t(x)=Eg_{T_\alpha(t)}(x)\le
\left(\frac{d}{2\pi}\right)^{d/2}e^{-\frac{d}{2}} |x|^{-d}.$$
\end{proof}

The standard reference book on general potential theory is the
 monograph \cite{BG}. For
 convenience of the reader we collect here the basic information with emphasis on
 what is known (and needed further on) about the $\alpha$-stable relativistic process.

In general potential theory  a very important role is played by
$\lambda$-resolvent (potential)  kernels, $\lambda>0$ , which are
defined as

$$U_\lambda(x,y)=\int_0^\infty e^{-\lambda t} p_t(x-y)dt, \ x,y\in \Rd.$$

If the defining integral above is finite for $\lambda=0$, the
corresponding kernel is called a potential kernel and will be
denoted by $U(x,y)$. For the relativistic process the potential
kernel is well defined for $d\ge 3$ but contrary to the stable or
Brownian case it is not expressible as an elementary function.
Recall that for the isotropic $\alpha$-stable process the
potential kernel is equal to $C|x-y|^{\alpha-d}$ for $d>\alpha$
and for the Brownian motion it is $C|x-y|^{2-d}$ for $d\ge 3$,
where $C$'s are appropriate constants. One can prove that the
relativistic potential kernel  could be written as a series
involving the Macdonald
 functions of different orders but this formula does not seem very useful.
 Nevertheless the asymptotic behaviour  of the potential kernel was   established in  \cite{G}, \cite{RSV}.

  \begin{equation}\label{0-potential}
  U(x-y)\approx  |x-y|^{\alpha-d}, |x-y|\le 1;\quad
   U(x-y)\approx  |x-y|^{2-d}, |x-y|\ge 1.
  \end{equation}
Note that they suggest that the process locally
 behaves like a stable one and globally like a Brownian motion.
Despite the fact  we do not know any simple form for the potential
kernel, a formula for the $1$-potential kernel is known (e.g. see
\cite{ByczRyzMal}): $$
  U_1(x)=C(\alpha,d)\,
   \frac{K_{(d-\alpha)/2}(|x|)}{  |x|^{(d-\alpha)/2}}\,,
  $$
  where $C(\alpha,d)= \frac{2^{1-(d+\alpha)/2} }{ {\Gamma(\alpha/2)\pi^{d/2}}}$.

    The {\it first exit time} of an (open)
   set  $D\subset {\Rd}$
   by the process $X_t$ is defined by the formula
   $$
   \tau_{D}=\inf\{t> 0;\, X_t\notin D\}\,.
   $$

  The basic object in potential theory of $X_t$ is the
  $\lambda$-{\it harmonic measure}  of the
  set $D$. It is defined by the formula:
  $$ 
  P_D^{\lambda}(x,A)=
  E^x[\tau_D<\infty; e^{-\lambda \tau_D} {\bf{1}}_A(X_{\tau_D})].
  $$
  The density kernel of  the measure $P_D^{\lambda}(x,A)$ (if it exists) is called the
  $\lambda$-{\it Poisson kernel} of the set $D$. If $\lambda=0$ the corresponding kernel will be
  denoted by $P_D(x,z)$ called {\it Poisson kernel} of the set $D$.

  Another fundamental object of potential theory is the {\it killed process} $X_t^D$
  when exiting the set $D$. It is defined in terms of sample paths up to time $\tau_D$.
  More precisely, we have the following "change of variables" formula:
  $$
  E^x f(X_t^D) =  E^x[t<\tau_D; f(X_t)]\,,\quad t>0\,.
  $$
  The density function of transition probability of the process $X_t^D$ is denoted
  by $p_t^{D}$. We have
  \begin{equation}
  p_t^{D}(x,y) = p_t(x-y) -
   E^x[t> \tau_D; p_{ t-\tau_D}(X_{\tau_D}-y)]    \,, \quad x, y \in {\Rd}\,.\label{density100}
  \end{equation}
  Obviously, we obtain
   $$
     p_t^{D}(x,y) \le p_t(x,y) \,, \quad x, y \in {\Rd}\,.
   $$

  $p_t^{D}$ is a strongly contractive semigroup (under composition) and shares most
  of properties of the semigroup $ p_t$. In particular, it is strongly Feller and
  symmetric: $ p_t^{D}(x,y) =  p_t^{D}(y,x)$.

  The $\lambda$-potential of the process $X_t^D$ is called the
  $\lambda$-{\it Green function} and is denoted by $G_D^{\lambda}$. Thus, we have
  $$
   G_D^{\lambda}(x,y)= \int_0^{\infty} e^{-\lambda t}\,p_t^{D}(x,y)\,dt\,.
  $$

  If $\lambda=0$ the corresponding kernel will be called {\it Green function} of the set $D$ and denoted  $G_D(x,y)$.

  Integrating (\ref{density100}) we obtain  for $\lambda >0$,
$$G^\lambda_{D}(x,y)=U_\lambda (x,y)-E^x\, e^{-\lambda\tau_{D}}U_\lambda (X_{\tau_{D}},y). $$
 Suppose that $D_1\subset D_2$ are two open sets. By the Strong Markov Property
\begin{eqnarray}
& &\!\!\!\! G^\lambda_{D_2}(x,y)-G^\lambda_{D_1}(x,y)\nonumber\\
& &\!= E^x\left[e^{-\lambda\tau_{D_1}}
U_\lambda(X_{\tau_{D_1}},y)-e^{-\lambda\tau_{{D_2}}}
U_\lambda(X_{\tau_{D_2}},y)\right]\nonumber\\& &\!=
E^x\left[\tau_{D_1}<\tau_{D_2};e^{-\lambda\tau_{D_1}}
\left(U_\lambda(X_{\tau_{D_1}},y)-e^{-\lambda\tau_{D_2}\circ\theta_{\tau_{D_1}}}U_\lambda(X_{\tau_{D_2}},y)\right)\right]\nonumber\\
& &\!=E^x\left[\tau_{D_1}<\tau_{D_2};e^{-\lambda\tau_{D_1}} \left(U_\lambda(X_{\tau_{D_1}},y)-E^{X_{\tau_{D_1}}}e^{-\lambda\tau_{D_2}}U_\lambda(X_{\tau_{D_2}},y)\right)\right]\nonumber\\
& &\! = E^x\left[\tau_{D_1}<\tau_{D_2};e^{-\lambda\tau_{D_1}}
G^\lambda_{D_2}(X_{\tau_{D_1}},y)\right].\label{greenpot100}
\end{eqnarray}

  The main purpose of the present paper is to obtain sharp estimates for the Green
  function for $D=\H=\{x\in \Rd:\, x_d>0\}$. The  investigation of
  Green functions of the relativistic process  for unbounded sets seems
  not to be treated in the literature. For bounded sets there many results obtained in recent
  years showing that the Green functions for open bounded sets under some assumptions about
  regularity of their boundary are comparable to their stable counterparts in
  $\Rd$, $d> \alpha$ (\cite{Ry}, \cite{CS}, \cite{KL}). That is,
  for $x,y\in D$,
  \begin{equation}\label{comparison}
  C(D)^{-1} G^{stable}_D(x,y)\le  G_D(x,y)\le C(D) G^{stable}_D(x,y),
  \end{equation}
  where $G_D^{stable}$ is the corresponding Green function for the isotropic stable process and  $C(D)$ is a constant usually dependent  on $\textrm{diam} (D)= sup_{x,y\in D}|x-y|$.
  Unfortunately in all known general bounds of the above type the dependence on
  the set $D$ in the constant $C(D)$ is not very clear and  $C(D)$ grows to $\infty$ with
  $\textrm{diam} (D)$. The constant also   depends on some other characteristics of $D$ as e.g.
  Lipschitz characteristic of $D$ when $D$ is a Lipschitz set.    Therefore  it is not
possible to use well known exact formulas or  estimates for the
stable Green functions of regular sets as  half-spaces, balls or
cones to derive the corresponding optimal estimates for the
relativistic process. Even for balls the constants grow to
$\infty$
 and (\ref{comparison}) does not yield any estimate
for a half-space in the limiting procedure.

Now suppose that $D$ is a  bounded set with a $C^{1,1}$ boundary.
It is well known that there is a $\rho>0$ such that for each point
$ z\in \partial D$ there are balls $B_z\subset D$,   $B^*_z
\subset D^c$ of radius $\rho$ such that $z\in \overline{B_z}
\cap\overline{B_z^*}$.  Denote by $\rho_0=\rho_0(D)$ the largest
$\rho$ having the above property. Finally let $\gamma=
\textrm{diam}\, D/\rho_0$. However not explicitly stated, the
following bound can be deduced from the results proved in
\cite{Ry}, for $d>\alpha$, $x,y\in D$:
\begin{eqnarray}\label{comparisonR}
  C_1(\gamma)C( \textrm{diam}(D))^{-1} G^{stable}_D(x,y)&\le&G_D(x,y)\nonumber\\
  &\le&  C_2(\gamma)\,C( \textrm{diam}(D))
  G^{stable}_D(x,y),\nonumber\\
  \end{eqnarray}
  where the constant $C$ can be chosen in such a way that $C(\textrm{diam}(D))=1$ for  $\textrm{ diam}(D)\le 1$ and
  $C(\textrm{diam}(D))$ increases with $\textrm{diam}(D)$. With some extra effort one can prove that the growth
   is polynomial. The constants  $C_1(\gamma), C_2(\gamma)$ can be chosen as continuous with respect to $\gamma$.
 Note that  if $D$ is a ball than we can take absolute constants (depending only on $\alpha$ and $d$) instead of $C_1(\gamma), C_2(\gamma)$.

Hence  for "smooth" sets with small or moderate diameter the
estimate (\ref{comparisonR}) is very satisfactory. For example for
balls of small or moderate diameter we obtain very precise
estimates using well known results for the isotropic stable
process. However, in the case of balls of large size, it would be
very interesting to find optimal estimates of the relativistic
Green function. Our main result  provides optimal estimates for
the Green function of the half-space $\H$. Also we found optimal
estimates for  intervals in $\mathbb{R}$. Despite the fact we do
not examine Green functions for balls   in higher dimensional
spaces we provide very precise estimates of the expected first
exit time from a ball.

Now we define harmonic and regular harmonic functions. Let $u$ be
a Borel measurable function on $\Rd$. We say that $u$ is {\em
harmonic} function in an open set $D\subset \Rd$ if
$$u(x)=E^xu(X_{\tau_B}), \quad x\in B,$$
for every bounded open set $B$ with the closure
$\overline{B}\subset D$. We say that $u$ is {\em regular harmonic}
if
$$u(x)=E^x[\tau_D<\infty; u(X_{\tau_D}))], \quad x\in D.$$

As a result of (\ref{comparisonR}) we obtain the following version
of the Boundary Harnack Principle (for details see \cite{Ry} or
\cite{GRy} in the  one-dimensional case).

\begin{thm}\label{BHP}[BHP] Let $D$ be a  bounded set with a
$C^{1,1}$ boundary. Suppose that   $diam\, D \le 4$ and
$\rho_0(D)\ge 1$. Let $ z\in \partial D$. If $f$ is a non-negative
regular harmonic function
 on D and $f(x)=0$, $x\in B(z,1)\cap D^c$. Then
$$f(x)\approx f(x_0)\delta_D(x)^{\alpha/2},\quad x\in  B(z,1/2), $$
where  $\delta_D(x)=dist(x, \partial D )$ and $x_0 \in D$ such
that $\delta_D(x_0)=1$.
\end{thm}

 For the purpose of this paper we state the  following specialized form of BHP which can be easily deduced from Theorem \ref{BHP}.
\begin{lem}    \label{BHP1}
Let $\H\ni\mathbf{1}=(0,\dots,0,1)$ and let  $F= B(0,
\sqrt{2})\cap \H$. Suppose that $f$ is a regular nonnegative
harmonic on $F$ such that $f(x)=0,\ x\in  \H^c$. Then for every $
x\in B(0, 1)\cap \H$ we have
$$f(x)\approx f(\mathbf{1})x_d^{\alpha/2}.$$
Assume that $R\ge 2$. Let $D=B(0, R)$,
 $z_0=(0,\dots,0,R)$ and $x_0=(0,\dots,0,R-1)$. Let  $F= B(z_0,
2)\cap D$. Suppose that $f$ is regular nonnegative harmonic on $F$
such that $f(x)=0,\ x\in D^c$. Then for every $ x\in B(z_0, 1)\cap
D$ we have
$$f(x)\approx f(x_0)(R-|x|)^{\alpha/2}.$$
\end{lem}
As mentioned above, the one-dimensional  case for intervals was
treated recently  in \cite{GRy} and since we will need it in the
next section we present it in a convenient form of the estimate of
the Poisson kernel.  Actually in \cite{GRy} it was shown that the
Green function of $(0,R)$ is comparable with
     the Green function of the corresponding stable process (with uniform constant for $R\le 3$). By standard arguments
      (see \cite{Ry}) this implies the  lemma below.
\begin{lem}    \label{GrzywnyRyznar} Assume that $d=1$ and $0<R\le 3$. Let $D=(0,R)$.  Then
     $$P_D ( x, z)\approx  \frac {(x(R-x))^{\alpha/2}}{(R(z-R))^{\alpha/2}(z-x)}\ e^{-z}
     , \quad x\in D,\ z>R.$$
     This implies that
    $$  P^x ( X_{\tau_D}>R)\approx (x/R)^{\alpha/2}, \quad x\in D$$
  and  
    $$ E^x [X_{\tau_D}>R;X_{\tau_D}]\approx (x/R)^{\alpha/2}((R-x)^{\alpha/2}+x),\quad x\in D.$$
    We also have that $$E^x\tau_D \approx {(x(R-x))^{\alpha/2}},\quad x\in D.$$
    \end{lem}

 Obtaining any exact formulas for the Green function or the Poisson kernel even for regular sets
 seems to be a very hard task but  in the recent paper \cite{ByczRyzMal} the formulas for  the 1-Poisson
 and 1-Green function of $\H$  were described explicitly in terms of the Macdonald functions:

 \begin{thm} \label{rel1Poisson}
 Let
 $$E^x[ e^{-\tau_\H},X_{\tau_\H}\in du]=P_{\H}^1(x,u)\,$$
be the 1-Poisson kernel for $\H$. Then we have
 $$
 P_{\H}^1(x,u)=2\frac {\sin (\alpha \pi/2) }{\pi} (2\pi)^{-d/2}
 \left(\frac{x_d}{ -u_d} \right)^{\alpha/2} \,
\frac{ K_{d/2}(|x-u|)}{ |x-u|^{d/2}},
 $$
 where $u_d<0<x_d$. Let  $G_{\H}^1(x,y)$ be the 1-Green function for $\H$
 then for $x,y\in \H$,
$$
 G_{\H}^1(x,y) = \frac{2^{1-\alpha} |x-y|^{\alpha-d/2}} { (2\pi)^{d/2} \Gamma(\alpha/2)^2}
 \int_0^{\frac{4 x_d y_d}{ |x-y|^2} } \frac{t^{\frac{\alpha}{ 2} -1}}{ (t+1)^{d/4}}
  K_{d/2}(|x-y|(t+1)^{1/2})  dt.
 $$
 Moreover,

 \begin{eqnarray}\label{intgral}
 \int_{\H}G_{\H}^1(x,y)dy&= &1-E^x e^{-\tau_\H}\nonumber\\&=& \frac{ 1}{ {\Gamma(\alpha/2)}} \int_0^{x_d}t^{\alpha/2-1} e^{-t}\,dt\,, \ x\in \H\,.
  \end{eqnarray}

  \end{thm}

  This result will be very useful in our analysis since, as shown in \cite{ByczRyzMal} the behaviour of the  Green function  $G_{\H}(x,y)$ could
  be described  in terms of the 1-Green function $G_{\H}^1(x,y)$   when $x$ and $y$ are close enough.

One of our main tools in establishing the upper bounds of the
Green function will be  estimates for the tail function
$P^x(\tau_\H> t)$. We start with the following lemma
     taken from the Master Thesis of the first author \cite{G}.
     \begin{lem}    \label{Grzywny}
     There is a constant C such that
\begin{equation} \label{Grzywny1}P^x(\tau_\H> t)\le
 C \frac{x_d + \ln (t+1)}{t^{1/2}}\,,\quad t\ge 1\,,\, x_d>0.\end{equation}
\end{lem}

\begin{proof}
Let $Y_t=X^{(d)}_t$, where $X_t=(X^{(1)}_t,\ldots,X^{(d)}_t)$. By
the symmetry of the random variable $Y_t$ we obtain
\begin{eqnarray*}P^x (\tau_\H>t)&=& P^x (\inf_{s\le t}Y_s >0)\\&=&  P^0
(\inf_{s\le t}(-Y_s+x_d) >0)= P^0 (\sup_{s\le t}Y_s <
x_d).\end{eqnarray*} Using a version of the L\'evy inequality
(\cite{B}, Ch.7, 37.9) we have for any
 $\varepsilon,y>0$ that
$$
2 P^0 (Y_t\ge
y+2\varepsilon)-2\sum^n_{k=1}P^0(Y_{\frac{tk}{n}}-Y_{\frac{t(k-1)}{n}}\ge
\varepsilon) \le P^0 (\sup_{k\le n}Y_{\frac{tk}{n}} \ge y).
$$
Note that
$\sum^n_{k=1}P^0(Y_{\frac{tk}{n}}-Y_{\frac{t(k-1)}{n}}\ge
\varepsilon)=nP^0(Y_{\frac{t}{n}}\ge \varepsilon)\to
t\int^{\infty}_\varepsilon\nu(x)dx$, hence, by symmetry again
\begin{eqnarray*} P^0
(\sup_{s\le t}Y_s \ge y)&\ge& 2 P^0 (Y_t\ge
y+2\varepsilon)-2t\int^{\infty}_{\varepsilon}\nu(x)dx\\&=&P^0
(|Y_t|\ge y+2\varepsilon)-2t\int^{\infty}_{\varepsilon}\nu(x)dx
.\end{eqnarray*}  This implies that
$$P^x (\tau_\H>t)=P^0 (\sup_{s\le t}Y_s < x_d)\le P^0 (|Y_t| <x_d+2\varepsilon)
+2t\int^{\infty}_{\varepsilon}\nu(x)dx\,. $$
 For $\varepsilon\ge
1$ we obtain from \pref{levymeasure} and \pref{asympt_infty}
$$\int^{\infty}_{\varepsilon}\nu(x)dx\le C e^{-\varepsilon}\varepsilon^{-\alpha/2-1}.$$
 Lemma \ref{transden_0} implies
 that the density of $Y(t)$ is bounded by $Ct^{-1/2}$, $t\ge 1$, hence taking
$\varepsilon=\frac{3}{2}\ln (t+1)$ we obtain
$$P^x (\tau_\H>t)\le
 C \left(x_d+\ln (t+1)\right) t^{-1/2}.$$
\end{proof}

    In order to improve the above estimate for  $x$ close to the boundary we use Lemma \ref{GrzywnyRyznar}.
    \begin{lem}  \label{Grzywnyimprove}
     For $0<x_d<2$ we have
     \begin{equation}\label{uboundtail}
     P^x(\tau_\H> t)\le C x_d^{\alpha/2}\;\ln (t+1)/t^{1/2},\quad t\ge 2,
     \end{equation}
     where $C$ is a constant.
     \end{lem}
     \begin{proof}
     It is enough to prove the claim for $d=1$. Let $D=(0,2)$ and assume that $0<x<2$.

     By the Strong Markov Property and then by  Lemma \ref{Grzywny} we obtain for $t\ge 1$:
     \begin{eqnarray*}
     P^x(\tau_{(0,\infty)}> 2t)&\le& P^x(\tau_D> t, \tau_{(0,\infty)}> 2t)\\& &
     +\,
     E^x[\tau_D<\tau_{(0,\infty)}\,; P^{X_{\tau_D}}(\tau_{(0,\infty)}> t)]\\
     &\le& P^x(\tau_D> t)\\& &+\,
     CE^x [\tau_D<\tau_{(0,\infty)}; X_{\tau_D} + \ln (t+1)]/t^{1/2}\\
     &\le& \frac{E^x\tau_D}{t}+
     CE^x [X_{\tau_D}>2\,;X_{\tau_D}]/t^{1/2}\\
     & &+\, C\ln (t+1)\;P^x (X_{\tau_D}>2)/t^{1/2}\\
     &\le&C x^{\alpha/2}\ln (t+1)/t^{1/2}.
    \end{eqnarray*}
     The last inequality follows from Lemma \ref{GrzywnyRyznar}. The proof is complete.
     \end{proof}
     The estimates from Lemmas \ref{Grzywny}, \ref{Grzywnyimprove} will
be very useful for the estimates of the Green function of the
half-line,
 however they are not optimal. In the sequel we will be able to improve them to
  be sharp enough and optimal (see Proposition \ref{optimaltail}). This will have a great  importance
  in estimating the Green function for a half-space in the $d$-dimensional case.

\begin{lem}    \label{density}  There is a constant $C$ such that for any  open set $D$:
$$ E g_{T_{\alpha}(t)} ^D(x,y) \le p_t^{D}(x,y)
\le  C (t^{-d/2}+ t^{-d/\alpha}) P^x(\tau_D> t/3)P^y(\tau_D>
t/3),$$ where $g_t ^D(x,y)$ is  the transition probability for the
Brownian motion killed on exiting $D$.
 \end{lem}
\begin{proof} We start with the upper bound.  Since $p_t^{D}(x,y)$ is a density of a
semigroup and  $p_t^{D}(x,y)\le \max_{z \in {\Rd}} p_t(z)$ then we
have
\begin{eqnarray*}p_{3t}^{D}(x,y)&=&\int_D \int_D
p_t^{D}(x,z)p^D_t(z,w)p_t^{D}(w,y)dz \,dw\\&\le& \max_{z \in
{\Rd}} p_t(z)\int_D p_t^{D}(x,z)dz\int_D p_t^{D}(w,y)dw\\&=&
\max_{z \in {\Rd}} p_t(z)P^x(\tau_D> t)P^y(\tau_D>
t),\end{eqnarray*} which proves the upper bound since
 $\max_{z \in {\Rd}} p_t(z)\le
  C (t^{-d/2}+ t^{-d/\alpha})$ by Lemma \ref{transden_0}.

To get the lower bound we use the subordination of the process  to
the Brownian motion: $X_t=B_{T_{\alpha}(t)}$.
Then \begin{eqnarray*}p_t^{D}(x,y)&=&P^x( B_{T_{\alpha}(t)}\in dy,
 B_{T_{\alpha}(s)}\in D, 0 \le s<t )\\&\ge&
P^x( B_{T_{\alpha}(t)}\in dy, B_s\in D, 0 \le s<T_{\alpha}(t)
).\end{eqnarray*}
Using the independence of $T_{\alpha}$ and the Brownian motion $B$
we obtain
$$P^x( B_{T_{\alpha}(t)}\in dy, B_s\in D, 0 \le s<T_{\alpha}(t)|T_{\alpha}(\cdot) )=
 g^D_{T_{\alpha}(t)}(x,y),$$
Integrating  we obtain the lower bound.
\end{proof}

The following lemma provides a very useful lower bound. Its proof
closely follows the approach used in  \cite{RSV}, where the bounds
on the potential kernels (Green functions for the whole $\Rd$)
were established for some special subordinated  Brownian motions
(in particular for  our process) for $d \ge 3$.
\begin{lem}    \label{potential_lower} For any open set $D\in {\Rd}$  we have
$$  G_D(x,y)\ge  \frac2\alpha G_{D}^{gauss}(x,y),$$
where  $G_{D}^{gauss}(x,y)$ is the Green function of $D$ for the
Brownian motion.
 \end{lem}
\begin{proof} Let $Q(x,y)=\int_0^\infty E g^D_{T_{\alpha}(t)}(x,y)dt$. From the previous
 lemma it is enough to prove that $Q(x,y)\ge  \frac2\alpha   G_{D}^{gauss}(x,y).$ We have
\begin{eqnarray*}
Q(x,y)&=&\int_0^\infty E g_{T_{\alpha}(t)} ^D(x,y)dt\\
&=&\int_0^\infty e^{t}\int_0^\infty g_u^D(x,y)e^{-u}\theta_\alpha(t,u) du dt\\
&=&\int_0^\infty g_u^D(x,y)e^{-u}\int_0^\infty e^{t}\theta_\alpha(t,u)dt du\\
&=&\int_0^\infty g_u^D(x,y)G(u) du,
\end{eqnarray*}
where  $G(u)=e^{-u}\int_0^\infty e^{t}\theta_\alpha(t,u)dt$ is the
potential kernel of the subordinator $T_\alpha(t)$. It was proved
in \cite{RSV} that $G(u)$ is
  a completely monotone (hence decreasing) function and
$\inf_{u> 0}G(u)=\lim_{u\to \infty}G(u)=C_\alpha$. 
 We find the constant $C_\alpha$ by taking into account
 the asymptotics of the
  Laplace transform of $G(u)$ at the origin:
\begin{eqnarray*}
\int_0^\infty e^{-\lambda u}G(u) du &=& \int_0^\infty
e^{t}\int_0^\infty e^{-u(1+\lambda)}\theta_\alpha(t,u)du dt \\&=&
\int_0^\infty e^{t}e^{-(1+\lambda)^{\alpha/2} t}dt=
 \frac 1 {(1+\lambda)^{\alpha/2} -1}\\&\cong& \frac 2{\lambda \alpha},\quad \lambda\rightarrow 0.
\end{eqnarray*}
Applying the monotone density theorem we obtain that $C_\alpha= 2/
\alpha$.
 Thus, since $g_u^D(x,y)\ge 0$, we finally obtain
\begin{eqnarray*}
Q(x,y)
&=&\int_0^\infty g_u^D(x,y)G(u) du\\
&\ge&\frac2\alpha  \int_0^\infty g_u^D(x,y) du = \frac2\alpha
G_{D}^{gauss}(x,y).
\end{eqnarray*}
\end{proof}

At this point let us recall that the exact formulas for the
Brownian Green  functions are well known for several regular sets
as balls or half-spaces (see e.g. \cite{Ba}). Since some of them
will be useful in the sequel we will list them for the future
reference. Recall that the Brownian motion we refer to in this
paper  has its clock running twice faster then the usual Brownian
motion. For the half-space $\H$, for $d\ge3$, we have that
 \begin{eqnarray}G^{gauss}_\H(x,y)&=&  C(d)\left[\frac1{|x-y|^{d-2}}- \frac1{|x-y^*|^{d-2}}\right]\nonumber\\
 \label{gaussGreen} &\approx& \min\left\{\frac{ x_d y_d }{|x-y|^d}, \frac1{|x-y|^{d-2}}\right\}, \quad x,y\in\H,
\end{eqnarray}
where $y^*=(y_1,\dots,y_{d-1}, -y_d)\in \H^c$. \\
For the half-space $\H$, for $d=2$,
 \begin{equation}\label{gaussGreen2}G^{gauss}_\H(x,y)= \frac1{2\pi} \ln\left(1+4\frac{x_2y_2}{|x-y|^2}\right), \quad x,y\in\H.
\end{equation}
In the one dimensional case
\begin{equation}\label{gaussGreen1}G^{gauss}_{(0,\infty)}(x,y)= x\wedge y, \quad x,y>0. \end{equation}
For the finite interval $(0, R)$ we have
 \begin{equation}\label{gaussGreen3}G^{gauss}_{(0, R)}(x,y)= \frac {x(R-y) \wedge y(R-x)}R, \quad x, y\in (0,R). \end{equation}

\begin{lem} \label{continuity}Let $D$ be an open  subset of $\H$.  For fixed $y\in \H  $ the function
$G_{\H}(\cdot,y)$ is  regular harmonic on $D$ provided $ y \notin
\overline{D}$. The same conclusion holds if $\H$ is replaced by an
open bounded set.
\end{lem}
\begin{proof}
The proof is standard and is  included  for completeness. First
observe that $G_{\H}(z,w)<\infty$ for $z\neq w$, which follows
from Lemmas \ref{transden2}, \ref{Grzywny} and \ref{density}.
Next, applying (\ref{greenpot100}) with $D_2=\H$ and $D_1=D$ we
have

\begin{equation}G^\lambda_{\H}(x,y)-G^\lambda_{D}(x,y)= E^x\left[\tau_{D}<\tau_{\H};e^{-\lambda\tau_{D}}G^\lambda_{\H}(X_{\tau_{D}},y)\right].\label{harmonic100}\end{equation}

 If $y\notin \overline{D}$ then $\tau_D=0$ and $X_{\tau_D}=y,$ $P^y$- a.s. so $G^\lambda_{D}(x,y)=G^\lambda_D(y,x)=0$. Moreover
 $E^x G^\lambda_{\H}(X_{\tau_{\H}},y)=0$, which follows from the fact that
$P^x\left(X_{\tau_{\H}} \in \H^c\setminus (\H^c)^r \right)=0$,
where $(\H^c)^r$ is a set of regular points of $\H^c$ and for
every $z\in (\H^c)^r$, $y\in \Rd$ we have $G^\lambda_{\H}(z,y)=0$
(see \cite{BG}). This implies that (\ref{harmonic100}) can be
rewritten as
$$G^\lambda_{\H}(x,y)= E^x\,e^{-\lambda\tau_{D}}G^\lambda_{\H}(X_{\tau_{D}},y).$$
Passing with $\lambda\to 0$ and observing that
$G^\lambda_{\H}\nearrow G_{\H}$ we obtain the conclusion by the
monotone convergence theorem.

The same arguments can be applied for any bounded set $F$, since
there is a half-space containing  $F$, which guarantees that the
Green function $G_{F}(x,y)< \infty$ for $x\neq y$.

\end{proof}

The estimates below  following from  Theorem \ref{rel1Poisson}
were  proved in \cite{BRB}. They turn out to be useful in the next
sections.
\begin{thm}\label{Green1est}
Assume that $d=1$ and $\alpha\ge 1$. When $|x-y|\ge 1\wedge
x\wedge y>0$ we obtain
$$G^1_{(0,\infty)}(x,y)\approx \frac{e^{-|x-y|}}{|x-y|^{1-\alpha/2}}(1\wedge x\wedge y)^{\alpha/2},$$
while for $|x-y|<1\wedge x\wedge y$ we obtain
\begin{eqnarray*}
G^1_{(0,\infty)}(x,y)&\approx& \log \left[2\frac{1\wedge x\wedge y}{|x-y|}\right],\quad \textrm{if }\quad \alpha=1,\\
G^1_{(0,\infty)}(x,y)&\approx& (1\wedge x\wedge
y)^{\alpha-1},\quad \textrm{if }\quad \alpha>1.
\end{eqnarray*}
In the remaining case, $\alpha<d$, we have
$$G^1_\H(x,y)\approx \frac{K_{(d-\alpha)/2}(|x-y|)}{|x-y|^{(d-\alpha)/2}}\left[\left(\frac{1\wedge x_d\wedge y_d}{|x-y|\wedge 1}\right)^{\alpha/2}\wedge 1\right].$$
\end{thm}
Finally we  state some basic scaling properties both for the
Poisson kernel and
  the Green function. The proof employs the scaling property \pref{scaling} and
  consists of elementary but tedious calculation hence is omitted.

  \begin{lem}[Scaling Property] \label{scaling1}
Let $D$ be an open subset of $\Rd$ and $P_{D,m}$,  $G_{D,m}$ be
the Poisson kernel, or the  Green function, respectively, for $D$
for the process with parameter $m$. Then
$$P_{D,m}(x,u)= m^{d/\alpha}P_{m^{1/\alpha}D}(m^{1/\alpha}x,m^{1/\alpha}u),\quad
x\in D, u\in D^c\,, $$
$$G_{D,m}(x,y)= m^{(d-\alpha)/\alpha}G_{m^{1/\alpha}D}(m^{1/\alpha}x,m^{1/\alpha}y),\quad x\in D, y\in D\,. $$
Thus, if $D$ is a cone with vertex at $0$ we obtain:
$$P_{D,m}(x,u)= m^{d/\alpha}P_{D}(m^{1/\alpha}x,m^{1/\alpha}u),\quad x\in D, u\in D^c\,, $$
$$G_{D,m}(x,y)= m^{(d-\alpha)/\alpha}G_{D}(m^{1/\alpha}x,m^{1/\alpha}y), \quad x\in D,
 y\in D\,. $$
   \end{lem}
Due to these scaling properties it is enough to investigate the
case $m=1$.

\section{Green function of half-line}
\setcounter{equation}{0}

In this section $d=1$ and the half-space $\H$ is a half-line, that
is $\H=(0,\infty)$.

\begin{lem}\label{trivialbound}
Assume that $|x-y|\le 3$. Then there is $ C = C(\alpha)$ such that
$$\left(1\wedge x\wedge  y \right)^{\alpha/2}\le C G^1_{{(0,\infty)}}(x,y).$$
\end{lem}
\begin{proof}
We use Theorem \ref{Green1est}. First let $\alpha \ge 1$ and
$|x-y|\ge 1\wedge x\wedge y>0$, then
\begin{eqnarray*}G^1_{(0,\infty)}(x,y)&\ge&
ce^{-|x-y|}{|x-y|^{\alpha/2-1}}(1\wedge x\wedge y)^{\alpha/2}\\
&\ge& ce^{-3}3^{\alpha/2-1}(1 \wedge x\wedge y)^{\alpha/2}.
\end{eqnarray*}
 Suppose that
$|x-y|< 1\wedge x\wedge y$. For $\alpha=1$,
$$G^1_{(0,\infty)}(x,y)\ge c \log \left[2\frac{1\wedge x\wedge y}{|x-y|}\right]\ge c\log 2\ge c (1\wedge x\wedge  y )^{\alpha/2}.$$
For $\alpha>1$,
$$G^1_{(0,\infty)}(x,y)\ge c (1\wedge x\wedge  y )^{\alpha-1}\ge c (1\wedge x\wedge  y )^{\alpha/2}.$$
Next, observe that $ K_{(1-\alpha)/2}(r)/r^{(1-\alpha)/2}$ is
decreasing. Therefore for $\alpha<1$ we obtain
\begin{eqnarray*}G^1_{(0,\infty)}(x,y)&\ge& c \frac{K_{(1-\alpha)/2}(|x-y|)}{|x-y|^{(1-\alpha)/2}}\left[\left(\frac{1\wedge x\wedge y}{|x-y|\wedge 1}\right)^{\alpha/2}\wedge
1\right]\\
&\ge & c\frac{K_{(1-\alpha)/2}(3)}{3^{(1-\alpha)/2}}\left(1\wedge
x\wedge y\right)^{\alpha/2}.\end{eqnarray*}
\end{proof}

\begin{thm}\label{green}For $x,y>0$,
$$G_{(0,\infty)}(x,y)\approx
 G^1_{(0,\infty)}(x,y)+(x\wedge y)\vee (x\wedge y)^{\alpha/2}.$$
\end{thm}

\begin{proof}
Throughout the whole proof we assume that $0<x\le y$. The proof
will rely on the estimates of $P^x(\tau_{(0,\infty)}> t)$ derived
in the previous section and the application of Lemma
\ref{density}.

 We proceed to estimate the Green function from above.
First we split the  integration 
\begin{eqnarray*}\int_0^\infty p_t^{{(0,\infty)}}(x,y)dt & =&
\int_0^{6} p_t^{{(0,\infty)}}(x,y)dt + \int_{6}^{\infty}
p_t^{{(0,\infty)}}(x,y)dt\\&=& V(x,y)+R(x,y).\end{eqnarray*}
We start with the estimation of the second integral. Due to Lemma
\ref{density}, \begin{eqnarray*}R(x,y)&=&3\int_{2}^{\infty}
p_{3t}^{{(0,\infty)}}(x,y)dt\\&\le& c\int_{2}^{\infty}
 P^x(\tau_{(0,\infty)}> t)P^y(\tau_{(0,\infty)}> t)\, \frac{dt}{ t^{1/2}}. \end{eqnarray*}
First consider the case $y< \sqrt{2}$. Then using
\pref{uboundtail} we have
$$R(x,y)\le C (xy)^{\alpha/2} \int_{2}^{\infty}  \frac {(\ln t)^2}{t} \frac{dt}{t^{1/2}}
\le C (xy)^{\alpha/2}.$$
If  $y\ge \sqrt{2}$, using (\ref{Grzywny1}) we estimate
\begin{eqnarray*}
R(x,y)&\le&c\int_{2}^{\infty}\left(P^y(\tau_{(0,\infty)}>
t)\right)^2\, \frac{dt}{t^{1/2}} \\&=& c \int_{2}^{y^2}
\left(P^y(\tau_{(0,\infty)}> t)\right)^2\, \frac{dt}{t^{1/2}}
 + c \int_{y^2}^\infty \left(P^y(\tau_{(0,\infty)}> t)\right)^2 \frac{dt}{t^{1/2}}\\
&\le& C\int_2^{y^2}\, \frac{dt}{t^{1/2}} +
C\int_{y^2}^\infty \left(\frac{y + \ln t}{t^{1/2}}\right)^2\, \frac{dt}{t^{1/2}}\\
&\le& C y.
\end{eqnarray*}
 Hence
\begin{equation}\label{upperbound}
G_{(0,\infty)}(x,y)\le  C  \left\{
\begin{array}{ll}
    V(x,y)+(x\,y)^{\alpha/2}, & \hbox{$y<1$,} \\
    V(x,y)+y, & \hbox{$y\ge 1$.} \\
\end{array}
\right.
 \end{equation}

 Let $B=(n+2,\infty)$, $n\in\mathbb{N}$. Now assume that   $n< x\le  n+1$ and $y\in B $.
   We claim that
   \begin{equation}\label{linBound}
 G_{{(0,\infty)}}(x,y)\le Cx, \end{equation}
 where  $C$ depends only on $\alpha$.  

Observe that,
\begin{eqnarray}\label{greenestlarge11}
\int_0^{\infty}V(x,y)dy&=& \int_{0}^{6}
\int_0^{\infty}p_{(0,\infty)}(t,x,y)dydt\nonumber\\&=&
\int_{0}^{6}P^{x}(\tau_{{(0,\infty)}}>t)dt\le 6.
\end{eqnarray}

Consider $h(v)=G_{{(0,\infty)}}(x,v),\ v\in B $. By Lemma
\ref{continuity} it is regular harmonic on  $B$.
Hence using the estimate (\ref{upperbound}) we obtain
  \begin{eqnarray*} G_{{(0,\infty)}}(x,y)&=& E^y G_{{(0,\infty)}}(x,X_{\tau_{B}})\\
  &=&  E^y \left[G_{{(0,\infty)}}(x,X_{\tau_{B}}) ;X_{\tau_{B}}\in (0,n+2)\right]\\
  &\le&  E^y V(x,X_{\tau_{B}})+ C(n+2).
\end{eqnarray*}
Integrating  $G_{{(0,\infty)}}(v,y)$ with respect to  $dv$  and
applying (\ref{greenestlarge11}) we obtain
\begin{equation}\label{greenintegral}\int_n^{n+1} G_{{(0,\infty)}}(v,y)dv \le 6 + C (n+2).\end{equation}
The final argument for proving (\ref{linBound}) will use Lemma
\ref{GrzywnyRyznar}. Take $D=(n-1,n+2),$ and recall that $y> n+2$
and $x \in (n,n+1)$ . Due to Lemma \ref{continuity} the Green
function  $G_{{(0,\infty)}}(u,y)$ is positive  regular harmonic on
$D$ as a function of  $u$.  By Harnack's inequality for harmonic
functions on $D$,  which follows from Lemma \ref{GrzywnyRyznar},
we arrive at
    $$G_{{(0,\infty)}}(x,y)\le C G_{{(0,\infty)}}(u,y), \ \ x,u\in (n,n+1),$$
which together with (\ref{greenintegral}) completes the proof of
the estimate
 $$ G_{{(0,\infty)}}(x,y) \le C x,\ \   1<x\le y-2.$$
Combining this with (\ref{upperbound}) we obtain
$$
G_{(0,\infty)}(x,y)\le C( V(x,y)+x),\quad  x\ge 1.
$$
 Since   $G^{gauss}_{(0,\infty)}(x,y)= x$ (see  (\ref{gaussGreen2})), then  by Lemma \ref{potential_lower} we  have that
$$G_{(0,\infty)}(x,y)\ge \frac2\alpha x.$$
Therefore we proved that
\begin{equation}\label{greenestlarge1}
G_{(0,\infty)}(x,y)\approx V(x,y)+x,\quad  x\ge 1.
\end{equation}
To estimate $V(x,y)$ we use
\begin{eqnarray*}\label{estimate111}V(x,y)&=&\int_0^{6} p_t^{{(0,\infty)}}(x,y)dt\nonumber \le e^6\int_0^{6}e^{-t} p_t^{{(0,\infty)}}(x,y)dt
 \\&\le& e^6\, G^1_{{(0,\infty)}}(x,y).
 \end{eqnarray*}
Next consider $x<1$ and $y\le 2$. By  (\ref{upperbound}) and Lemma
\ref{trivialbound} we get $$
G_{{(0,\infty)}}(x,y)\approx G^1_{{(0,\infty)}}(x,y)\approx
G^1_{{(0,\infty)}}(x,y)+ x^{\alpha/2}.
$$
Now assume that  $x<1$ and $y>2$. Again
$G_{{(0,\infty)}}(\cdot,y)$, by Lemma \ref{continuity},  is
regular harmonic on $(0,2)$, hence by BHP (see Lemma  \ref{BHP1}):
$$G_{{(0,\infty)}}(x,y)\approx G_{(0,\infty)}(1,y)x^{\alpha/2}.$$
Due to Theorem \ref{Green1est}, $G^1_{{(0,\infty)}}(1,y)\le C$ so
by  (\ref{greenestlarge1})    we have
$$G_{(0,\infty)}(1,y) \approx 1,$$
which implies
$$G_{{(0,\infty)}}(x,y)\approx x^{\alpha/2},\quad x<1,  y>2.$$
This completes the proof.
\end{proof}


\begin{rem}\label{Greenhalflineasymp} Let $x\le y$. Then we have
\begin{equation}
G_{(0,\infty)}(x,y)\approx \left\{
                  \begin{array}{ll}
                    G^1_{(0,\infty)}(x,y), & \hbox{$x\le 1$, $|x-y|<1$;} \\
                    G^1_{(0,\infty)}(x,y)+x, & \hbox{$x>1$, $|x-y|<1$ ;} \\
                    x\vee x^{\alpha/2}, & \hbox{$|x-y|\ge 1$.}
                  \end{array}
                \right.
\end{equation}
\end{rem}

\section{Exit time properties}
\setcounter{equation}{0}

In this section we derive optimal estimates of the expected value
of the exit time from a ball of arbitrary radius. Then, which
seems the most important result of this section, we provide
optimal estimates of the tail distribution for the exit time from
a half-space. That is we improve the bounds obtained in Lemmas
\ref{Grzywny} and \ref{Grzywnyimprove}. They will play a crucial
role in the next section, where we deal with the Green function of
a half-space in $\Rd$. We start with the one-dimensional case.
\begin{prop}\label{exittime}
For $x\in (0,R)$ we have
$$
E^x\tau_{(0,R)}\approx (x^{\alpha/2}\vee x)
\left((R-x)^{\alpha/2}\vee (R-x)\right).
$$
\end{prop}
\begin{proof}
If $R\le 3$ then from Lemma \ref{GrzywnyRyznar} we have
$$E^x\tau_{(0,R)}\approx (x(R-x))^{\alpha/2}.$$

Throughout the rest of the proof we  suppose that $R>3$. Assume
$x\le R/2$. First we prove the upper bound. By Theorem \ref{green}
and \pref{intgral} we obtain
\begin{eqnarray} E^x\tau_{(0,R)}
&=&\int^R_0 G_{(0,R)}(x,y)dy\le\int^R_0 G_{(0,\infty)}(x,y)dy\nonumber\\
&\approx& \int^R_0 G^1_{(0,\infty)}(x,y)dy +\int^x_0
(y^{\alpha/2}\vee y) dy + (R-x)(x^{\alpha/2}\vee
x)\nonumber\\&\le&2(\alpha\Gamma(\alpha/2))^{-1}x^{\alpha/2}+x(x^{\alpha/2}\vee
x)+(R-x)(x^{\alpha/2}\vee x)\nonumber\\&\le& c R (x^{\alpha/2}\vee
x). \label{exittimeproof1}
\end{eqnarray}

 Now, we deal with the lower bound. By Lemma \ref{potential_lower},
$$G_{(0,R)}(x,y)\ge \frac2\alpha G^{gauss}_{(0,R)}(x,y).$$
Denote the first exit time of $(0,R)$ for the Brownian motion by
$\tau^{gauss}_{(0,R)}$. It is well known that
$E^x\tau^{gauss}_{(0,R)}=\frac12\,x\,(R-x)$ (eg. see \cite{Du}).
Then we have \begin{eqnarray*}E^x\tau_{(0,R)} &=&\int_0^R
G_{(0,R)}(x,y)dy \ge \frac2\alpha\int_0^R G^{gauss}_{(0,R)}(x,y)
dy\\ &=& \frac2\alpha
E^x\tau^{gauss}_{(0,R)}=\frac1\alpha\,x\,(R-x).\end{eqnarray*}
Hence we get, for $1\le x\le R/2$,
\begin{equation}\label{exittimeproof2}
E^x\tau_{(0,R)}\approx x\,R
\end{equation}
Let $x<1$. Notice that by the Strong Markov Property
$$E^x\tau_{(0,R)}=s(x)+E^x\tau_{(0,2)},$$
where $s(x)=E^x(E^{X_{\tau_{(0,2)}}}\tau_{(0,R)})$ is regular
harmonic on the interval $(0,2)$ vanishing on its complement.
Therefore by BHP (see Lemma \ref{BHP1}) we obtain
$$s(x)\approx s(1)x^{\alpha/2}.$$ Moreover due to Lemma \ref{GrzywnyRyznar} we have
$$E^x\tau_{(0,2)} \approx x^{\alpha/2}.$$
 This yields 
$$E^x\tau_{(0,R)} \approx  (s(1) +1)
x^{\alpha/2}.$$ Noting that $s(1)=E^1\tau_{(0,R)}-E^1\tau_{(0,2)}$
and observing  that  \pref{exittimeproof2} implies
$$(E^1\tau_{(0,R)}-E^1\tau_{(0,2)})+1 \approx R,$$we obtain
\begin{equation}\label{exittimeproof3}E^x\tau_{(0,R)}\approx  R x^{\alpha/2},\quad  0< x<1.
\end{equation}
Putting together \pref{exittimeproof1}, \pref{exittimeproof2} and
\pref{exittimeproof3} we obtain
$$
E^x\tau_{(0,R)}\approx  R\, (x^{\alpha/2}\vee x),\quad \textrm{
for } x\le R/2.
$$
 By symmetry we have $E^x\tau_{(0,R)}=E^{R-x}\tau_{(0,R)}$, which ends the
 proof.
\end{proof}

Now we derive bounds for the expected exit times from balls  in
the  multidimensional case.

\begin{prop}\label{exittimeRd}
For $x\in B(0,R)=\{v\in \Rd:|v|<R\}$ we have
$$
 E^x\tau_{B(0,R)}\approx
\left((R-|x|)^{\alpha/2}\vee (R-|x|)\right)\left(R\vee
R^{\alpha/2}\right).
$$
\end{prop}
\begin{proof}  Let $\tau^{stable}_{B(0,R)}$ be the first exit time from $B(0,R)$ for the
$\alpha$-stable isotropic  process. By the result of Getoor
\cite{Ge} we have
$E^x\tau^{stable}_{B(0,R)}=c(R^2-|x|^2)^{\alpha/2}$ for
$c=c(\alpha,d)$.

 First assume that $R\le 3$. Then from \pref{comparisonR}  we obtain
$$E^x\tau_{B(0,R)}\approx E^x\tau^{stable}_{B(0,R)}=c(R^2-|x|^2)^{\alpha/2}\approx (R-|x|)^{\alpha/2}\,R^{\alpha/2},$$
which completes the proof in this case.

Next suppose that $R>3$. Let $z = x/|x|$ if $x\neq 0$ and
$z=(1,0,\ldots,0)$ if $x=0$. We now take $S_R= \{v:|\left\langle
z,v\right\rangle|<R\}$. The process $\left\langle
z,X_t\right\rangle$ is the one-dimensional relativistic process
(with the same parameter) which starts from $|x|$. Note that
$$E^x\tau_{B(0,R)}\le E^x\tau_{S_R}.$$ By the one-dimensional result (see Lemma \ref{exittime}) we get the upper bound.

For $|x|\le R-1$ we get the lower bound by using Lemma
\ref{potential_lower} and the result for the Brownian motion:
$E^x\tau^{gauss}_{B(0,R)}=\frac1{2d}(R^2-|x|^2)$ (see \cite{Du}).
Namely
\begin{eqnarray} E^x\tau_{B(0,R)}&=&\int_{B(0,R)}G_{B(0,R)}(x,y)dy\ge  \frac2{\alpha}\int_{B(0,R)}G^{gauss}_{B(0,R)}(x,y)dy\nonumber\\
&=& \frac2{\alpha} E^x\tau^{gauss}_{B(0,R)}= \frac1{\alpha
d}(R^2-|x|^2).\label{lowerextG} \end{eqnarray}
 To complete the proof
we need to consider $R-1\le |x|\le R$. The conclusion  will follow
in the usual way from BHP (see Lemma \ref{BHP1}) and the bound
above  for $|x|= R-1$. We may and do assume that
$x=(0,\dots,0,|x|)$. Denote  $x_0= (0,\dots,0,R-1)$ and $z_0=
(0,\dots,0,R)$.  Let $$F= B(0,R)\cap B(z_0,2)$$ and
$$s(x)=E^xE^{X(\tau_F)}\tau_{B(0,R)}.$$
Observe that $s(x)$ is a positive regular harmonic function on $F$
satisfying the assumptions of the second part of Lemma \ref{BHP1}
hence
$$ s(x)\approx s(x_0) (R-|x|)^{\alpha/2}.$$
Next, by the Strong Markov Property
\begin{eqnarray}E^x\tau_{B(0,R)}&=&s(x)+E^x\tau_{F}\ge s(x)+E^x\tau_{B(x_0,1)}\nonumber\\
&\approx&
s(x_0) (R-|x|)^{\alpha/2}+E^x\tau^{stable}_{B(x_0,1)}\nonumber\\
&\approx&(s(x_0)+1) (R-|x|)^{\alpha/2}\nonumber\\
&=&(E^{x_0}\tau_{B(0,R)}-E^{x_0}\tau_F+1)(R-|x|)^{\alpha/2}\nonumber\\
&\ge& c \,R (R-|x|)^{\alpha/2}\label{exitmefinal} .
\end{eqnarray}
The equivalence $E^x\tau_{B(x_0,1)} \approx
E^x\tau^{stable}_{B(x_0,1)}$ follows from (\ref{comparisonR}) and
$$E^{x_0}\tau_{B(0,R)}-E^{x_0}\tau_F+1 \ge c R$$ follows from
(\ref{lowerextG}). Combining  (\ref{lowerextG}) and
(\ref{exitmefinal}) we arrive at the desired lower bound.
\end{proof}

Now, we recall the Ikeda-Watanabe formula \cite{IW} which provides
a relationship between the Green function and the Poisson kernel.
Assume $D\subset \Rd$ is a nonempty open set and $E$ is a Borel
set such that $\textrm{dist}(D,E)>0$, then we have
\begin{equation}\label{Ikeda-WatanabeFormula}P^x(X(\tau_D)\in E, \tau_D<\infty)=\int_D
G_D(x,y)\nu(E-y)dy,\ x\in D.\end{equation} The following
generalization of the Ikeda-Watanabe formula was proved in
\cite{KS}:
\begin{equation}\label{Ikeda-WatanabeFormula2}
P^x(X(\tau_D)\in E,t_1<\tau_D<t_2)=\int_D
\int^{t_2}_{t_1}p^D_t(x,y)dt\nu(E-y)dy,
\end{equation}where $ 0\le t_1<t_2$, $x\in D$.
For $D$ which satisfies the outer cone property
  we have $P^x(X_{\tau_D}\in\partial D, \tau_D<\infty)=0$ (see \cite{KS}).
   Therefore the above formulas are  true for all sets $E\subset
   D^c$ for such $D$. In particular,  for sets studied in this paper  as  balls or half-spaces,
   the process does not hit the boundary, when exiting a set.

As a consequence of formula \pref{Ikeda-WatanabeFormula} we have
the following lemma which proof is omitted.
\begin{lem}\label{poissonKernel}
Let $D\subset \Rd$ be a bounded open set then
$$P_D(x,z)\le E^x \tau_D \sup_{v\in D}\nu(z-v),\ z\in (\overline{D})^c, x\in D.$$
Moreover, if $dist(z,D)\ge 1$ then
$$P_D(x,z)\le C E^x \tau_D e^{-dist(z,D)}.$$
\end{lem}

\begin{prop} \label{MRestimate}For $0<x<R$ we have
$$P^x(\tau_{(0,R)}<\tau_{(0,\infty)})
\approx  \frac{x^{\alpha/2}\vee x}{R^{\alpha/2}\vee R}.$$
\end{prop}
\begin{proof} Assume that $R\ge 1$ and $0<x<R$. By Lemma \ref{continuity} the function
$G_{(0,\infty)}(\cdot,2R)$ is regular harmonic on $(0,R)$,
therefore by Remark \ref{Greenhalflineasymp}  we obtain
\begin{eqnarray}\label{MRp1}C x^{\alpha/2}\vee x &\ge&
G_{(0,\infty)}(x,2R)=
E^xG_{(0,\infty)}(X_{\tau_{(0,R)}},2R)\nonumber\\&\ge& c
E^x[X_{\tau_{(0,R)}}\wedge 2R; X_{\tau_{(0,R)}}>0]\nonumber\\&\ge&
c R P^x(X_{\tau_{(0,R)}}>0).\end{eqnarray}

Let $n\ge 3$, which we specify later. Again
$G_{(0,\infty)}(\cdot,nR)$ is regular harmonic on $(0,R)$.
Applying   Remark \ref{Greenhalflineasymp} we have
 \begin{eqnarray}\label{MRp2}
G_{(0,\infty)}(x,nR)&=&
E^xG_{(0,\infty)}(X_{\tau_{(0,R)}},nR)\nonumber\\&\le& C\left(n R
P^x(X_{\tau_{(0,R)}}>0) +
E^xG^1_{(0,\infty)}(X_{\tau_{(0,R)}},nR)\right).\nonumber\\
\end{eqnarray}
Moreover  Lemma \ref{poissonKernel} and Theorem \ref{Green1est}
imply
\begin{eqnarray} & &\!\!\!
E^xG^1_{(0,\infty)}(X_{\tau_{(0,R)}},nR)\nonumber\\& &=\, \int_{R}^\infty G^1_{(0,\infty)}(v,nR)P_{(0,R)}(x,v)dv\nonumber\\
& &=\, \int_{R}^{(n-1)R}
G^1_{(0,\infty)}(v,nR)P_{(0,R)}(x,v)dv\nonumber\\& & \;+\,
\int_{(n-1)R}^\infty G^1_{(0,\infty)}(v,nR)P_{(0,R)}(x,v)dv\nonumber\\
& &\le\, \sup_{R\le v\le (n-1)R}
G^1_{(0,\infty)}(v,nR)P^x(X_{\tau_{(0,R)}}\in (R, (n-1)R)
)\nonumber\\ & &\;+
 \sup_{ v\ge {(n-1)R}}P_{(0,R)}(x,v) \int_{(n-1)R}^\infty G^1_{(0,\infty)}(v,nR)dv\nonumber\\
& &\le\, c P^x(X_{\tau_{(0,R)}}>0)+
Ce^{-(n-2)R}E^x\tau_{(0,R)},\label{MRp22}
\end{eqnarray}
where  $P_{(0,R)}(x,v)$ is the Poisson kernel for $(0,R)$ and by
Lemma \ref{poissonKernel} it admits
$$P_{(0,R)}(x,v)\le C E^x\tau_{(0,R)} e^{-(n-2)R},\quad v\ge (n-1)R.$$
  Using the  \pref{MRp2} and  \pref{MRp22} we arrive at
$$\label{MRp3}c \, n \,R\,P^x(X_{\tau_{(0,R)}}>0)\ge G_{(0,\infty)}(x,nR)- Ce^{-(n-2)R}E^x\tau_{(0,R)}.
$$
By Lemma \ref{exittime}, $E^x\tau_{(0,R)}\approx R\,
(x^{\alpha/2}\vee x) $, so  Remark \ref{Greenhalflineasymp}
implies
 $$G_{(0,\infty)}(x,n R)- Ce^{-(n-2)R}E^x\tau_{(0,R)}\ge (c-C R e^{-(n-2)R})\, (x^{\alpha/2}\vee x).$$
 Now we pick $n$ independently of $R\ge 1$ and large enough so
 that\\ $c-C R e^{-(n-2)R}\ge c/2$.
This yields
\begin{equation}\label{MRp4}
P^x(\tau_{(0,R)}<\tau_{(0,\infty)})\ge c\frac{x^{\alpha/2}\vee
x}{R}.
\end{equation}
Next, for $R<1$ we use Lemma \ref{GrzywnyRyznar} to get
\begin{equation}\label{MRp6}P^x(\tau_{(0,R)}<\tau_{(0,\infty)}) \approx
\frac{x^{\alpha/2}}{R^{\alpha/2}}.\end{equation} Combining
\pref{MRp1}, \pref {MRp4}  and \pref{MRp6} ends the proof.
\end{proof}

Now we can prove the main result of this section.
\begin{thm}\label{optimaltail} For $x>0$ and $t\ge 1$,
$$ P^x(\tau_{(0,\infty)}>t)\approx  \frac{x^{\alpha/2}\vee x}{t^{1/2}}\wedge 1.$$
\end{thm}
\begin{proof} Assume that $t\ge 1$. If $2x\ge t^{1/2}$ the upper bound is trivial. So we may assume that $2x< t^{1/2}$.
We have
$$P^x(\tau_{(0,\infty)}>t)\le P^x(\tau_{(0,R)}>t)+P^x(\tau_{(0,R)}<\tau_{(0,\infty)}).$$
Let $R>2x$. By Chebyschev's inequality and Proposition
\ref{exittime} we obtain $$P^x(\tau_{(0,R)}>t)\le
\frac{E^x\tau_{(0,R)}}{t}\approx \frac{(R\vee
R^{\alpha/2})(x^{\alpha/2}\vee x)}{t}.$$ By the Lemma
\ref{MRestimate}
$$P^x(\tau_{(0,R)}<\tau_{(0,\infty)})\le c \frac{x^{\alpha/2}\vee x}{R^{\alpha/2}\vee R }.$$
Setting $R=t^{1/2}$ we arrive at the upper bound.

Next, let us observe that by Lemma \ref{density},
\begin{eqnarray}\label{optailp1}P^x(\tau_{(0,\infty)}>t)&=&\int^{\infty}_0 p^{(0,\infty)}_t(x,y)dy\ge
\int^{\infty}_0 Eg^{(0,\infty)}_{T_\alpha(t)}(x,y)dy\nonumber\\&
=& E\, P^x(\tau^{gauss}_{(0,\infty)}>T_{\alpha}(t)).\end{eqnarray}
Let us observe that by Chebyschev's inequality and \pref{subord2}
for $\lambda=-1$ we have
$$P(T_{\alpha}(t)> 2t)=P(e^{T_{\alpha}(t)}>e^{2t})\le Ee^{T_\alpha(t)}e^{-2t}=e^{t-2t}\le e^{-1}.$$
That is $P(T_{\alpha}(t)\le 2t)\ge 1-e^{-1}$.
 Taking into account the fact that
$P^x(\tau^{gauss}_{(0,\infty)}>t)\approx \frac{x}{t^{1/2}}\wedge
1$ we obtain from \pref{optailp1},
 $$P^x(\tau_{(0,\infty)}>t)\ge c\,  E
 \left(\frac{x}{T_\alpha(t)^{1/2}}\wedge1\right)
 \ge c \left(\frac{x}{t^{1/2}}\wedge
1\right).$$

 Now, let $x< 1$ then
\begin{eqnarray*}P^x(\tau_{(0,\infty)}>t)&\ge& E^x \left[ X_{\tau_{(0,2)}}>0 ;
P^{X_{\tau_{(0,2)}}}(\tau_{(0,\infty)}>t)\right]\\ &\ge& c E^x
\left[ X_{\tau_{(0,2)}}\ge 2;
\frac{X_{\tau_{(0,2)}}}{t^{1/2}}\wedge1 \right]\\&\ge& c
\left(\frac{1}{t^{1/2}}\wedge 1\right) P^x(X_{\tau_{(0,2)}}\ge
2).\end{eqnarray*} Hence by Lemma \ref{MRestimate} we obtain
$$P^x(\tau_{(0,\infty)}>t)\ge c x^{\alpha/2}\frac{1}{t^{1/2}},$$
which completes the proof.
\end{proof}

\begin{cor}\label{density1A}
There exists a constant C such that, for $t>0$ and $x,y \ge1$,
\begin{equation}\label{density1A1}p_t^{(0,\infty)}(x,y) \le C (t^{-1/2}+t^{-1/\alpha})
\left(\frac{ x}{t^{1/2}}\wedge 1\right)\left(\frac{
y}{t^{1/2}}\wedge 1 \right).\end{equation} For $t\ge1$ and
$x,y>0$,
$$
c t^{-1/2} \left(\frac{ x}{t^{1/2}}\wedge 1\right)\left(\frac{
y}{t^{1/2}}\wedge 1\right)e^{-\frac{|x-y|^2}{c_1 t}} \le
p_t^{(0,\infty)}(x,y),$$ where $c$, $c_1$ are some constants.
 Hence, for $x,y,t\ge 1$, satisfying  $t\ge |x-y|^2$ we have the optimal bound
$$
p_t^{(0,\infty)}(x,y)\approx t^{-1/2} \left(\frac{
x}{t^{1/2}}\wedge 1\right)\left(\frac{ y}{t^{1/2}}\wedge 1
\right).
$$
\end{cor}
\begin{proof}
The upper bound immediately follows from Lemma  \ref{density} and
Theorem \ref{optimaltail}.

Pick $0<\beta<1/2$ such that $(1+ 1/\beta)^{\alpha/2}-2=1$ and let
$$A_t=\{\omega:\beta t <T_{\alpha}(t)(\omega)<2t\}.$$ To obtain the lower
bound we use again   Lemma  \ref{density} to get
 $$p_t^{(0,\infty)}(x,y) \ge E \left[g^{(0,\infty)}_{T_{\alpha}(t)} (x,y); A_t \right]. $$

Next by a classical result  \begin{eqnarray*}
g^{(0,\infty)}_t(x,y)&=& g_t(x-y)-g_t(x+y)
= g_t(x-y)\left(1-e^{-\frac{xy}{t}}\right)\\
&\ge& g_t(x-y)\left(1\wedge \frac{xy}{t}\right)\ge
g_t(x-y)\left(\frac{ x}{t^{1/2}}\wedge 1\right)\left(\frac{
y}{t^{1/2}}\wedge 1 \right).
\end{eqnarray*}
Hence
$$p_t^{(0,\infty)}(x,y)\ge c t^{-1/2}e^{-\frac{|x-y|^2}{4\beta t}}P(A_t)
\left(\frac{ x}{t^{1/2}}\wedge 1\right)\left(\frac{
y}{t^{1/2}}\wedge 1 \right).$$ Next we estimate $P(A^c_t)$.
 By Chebyschev's inequality and by \pref{subord2} for
$\lambda=1/\beta$,
\begin{eqnarray*}  P(T_{\alpha}(t)<\beta t)&=& P(e^{-(1/\beta)T_{\alpha}(t)}>e^{- t})\le e^{ t}Ee^{-(1/\beta)T_{\alpha}(t)}
   \\&=& e^{-((1+ 1/\beta)^{\alpha/2}-2) t}= e^{ -t} .\end{eqnarray*}
Similarly by \pref{subord2} for $\lambda=-1$,
  $$  P(T_{\alpha}(t)> 2 t)= P(e^{T_{\alpha}(t)}>e^{2t})\le e^{-2t}Ee^{T_{\alpha}(t)}
   = e^{ -t}.$$
   Hence
  $$P(A^c_t)\le 2e^{- t},$$
  which implies $\inf_{t>1}P(A_t)\ge 1-2e^{- 1}$ and this ends the proof.

  \end{proof}

  One of the drawbacks of the  inequality in the above Corollary
   is that the right hand side does not depend on the distance $|x-y|$.
   The following result will be very useful in the next section
    and it does take into account the distance $|x-y|$.

  \begin{thm}\label{density1}
Let $x,y\ge 1$ and $|x-y|\ge 1$. Then, for $t\le |x-y|^2$,
$$p^{(0,\infty)}_t(x,y)\le
    C\left(\frac{xy}{|x-y|^{2}}\wedge1\right)\left(g_{t}\left(c(x-y)\right),
    +t\,\nu\left(c(x-y)\right)\right)$$
where $c=8\sqrt{2}$ and $C$ is some constant. Moreover
$$\int^{\infty}_1t^{-d/2+1/2}p^{(0,\infty)}_t(x,y)dt\le c(d,\alpha)\frac{xy}{|x-y|^d}.$$
\end{thm}
\begin{proof} Our arguments are based on  the idea of  proof of Theorem 4.2 in \cite{KS}.
Throughout the whole proof we assume that  $x,y\ge 1$ and  $x\le
y-1$.

We first consider the case $t\le |x-y|^2/16$.  The interval
$(0,(x+y)/2)$ we denote
 by $S$ and $(y-s,y+s)$ by $D(s)$. Let $0<s<1/8$,
then $D(s)\in (0,\infty)\setminus S$. By the Strong Markov
Property we obtain
\begin{eqnarray}
& &\!\!\!\!\!\! \int_{D(s)}p^{(0,\infty)}_t(x,y)dz\nonumber\\&
&\!\!\!=\,P^x(X_t\in D(s),\tau_{{(0,\infty)}}>t)\nonumber\\&
&\!\!\!\le\,
P^x(\tau_S<t, X_{\tau_S}>0, X_t\in D(s))\nonumber\\
& &\!\!\!=\,E^x\left[ \tau_S<t,X_{\tau_S}\in(0,\infty)\setminus S,
P^{X(\tau_S)}(X_{t-\tau_S}\in D(s)) \right]. \label{estdensH3}
\end{eqnarray}

Let $A=(y-|x-y|/4,y+|x-y|/4)$ and $B=(0,\infty)\setminus(S\cup
A)$. Observe that $\textrm{dist}(A,S)=|x-y|/4$ and
$\textrm{dist}(B,D(s))\ge |x-y|/8$. Because $p_t(x)$ is radially
decreasing in $|x|$  we have for $X_{\tau_S}\in B$,
\begin{eqnarray*}
P^{X(\tau_S)}(X_u\in
D(s))&=&\int_{D(s)}p_u(X_{\tau_S}-z)dz\\
&\le&|D(s)|p_u\left(\frac{x-y}{8}\right)\\&\le&
c|D(s)|\left(g_{u}\left(\frac{x-y}{8\sqrt{2}}\right)+u\nu\left(\frac{x-y}{8\sqrt{2}}\right)\right),
\end{eqnarray*}
where in the last step we applied Lemma \ref{transden2}.
 Next  observe that $g_t(x)$ is an increasing function in $t$  on the interval $(0,x^2/2)$.
 Hence, for $t\le |x-y|^2/264$, we obtain, for $X_{\tau_S}\in B $,
$$P^{X(\tau_S)}(X_{t-\tau_S}\in
D(s))\le c|D(s)|\left(g_{t}\left(\frac{x-y}{8\sqrt{2}}\right)+
t\nu\left(\frac{x-y}{8\sqrt{2}}\right)\right) .$$ Define
$F(t,z)=g_{t}(z/(8\sqrt{2}))+ t\nu(z/(8\sqrt{2}))$. Then
Proposition \ref{MRestimate} and the above estimate yield
\begin{eqnarray}\label{estdensH4}
& &\!\!\!E^x\left[ \tau_S<t,X_{\tau_S}\in
B,P^{X({\tau_S})}(X_{t-\tau_S}\in D(s))
 \right]\nonumber\\& & \le\, c|D(s)|F(t,x-y)P^x(
\tau_S<t,X_{\tau_S}\in B)\nonumber\\
& &\le\, c|D(s)|F(t,x-y)P^x( \tau_S<\tau_{(0,\infty)})\nonumber\\
& &\le\, c|D(s)|F(t,x-y)\frac{x}{x+y}\nonumber\\
& &\le\, c|D(s)|F(t,x-y)\frac{xy}{|x-y|^2}.
\end{eqnarray}

For the set $A$ we have by \pref{Ikeda-WatanabeFormula2},
\begin{eqnarray*}
& &\!\!\!E^x\left[\tau_S<t,X_{\tau_S}\in
A,P^{X(\tau_S)}(X_{t-\tau_S}\in D(s))
 \right]  \\
& & = \int_{S}\int^t_0 p^S_r(x,z)\int_A\nu(z-w)P^w(X_{t-r}\in
 D(s))dwdrdz.
\end{eqnarray*}
  Moreover
\begin{eqnarray}\label{simple}\int_{A}P^w(X_{t}\in D(s))dw &=& \int_{A}\int_{D(s)}p(t,w,z)dzdw\nonumber\\&=&
 \int_{D(s)} \left(\int_{A}p(t,w,z)dw\right)dz\nonumber\\
 &=&
 \int_{D(s)}P^w(X_t\in A)dw\le |D(s)|.
 \end{eqnarray}
Using (\ref{simple}) and observing that  $\nu(z-w)\le
\nu((x-y)/4), \ w\in A, z\in S  $
 we obtain
\begin{eqnarray}\label{estdensH5}
& & E^x\left[ \tau_S<t,X_{\tau_S}\in
A,P^{X(\tau_S)}(X_{t-\tau_S}\in D(s))
 \right]\nonumber\\& &\le\,c|D(s)|\nu((x-y)/4)\int_S \int^{t}_0
 p_S(r,x,z)drdz\nonumber\\
 & &=\, c|D(s)|\nu((x-y)/4) \int^{t}_0
 P^x(\tau_S>r)dr\nonumber\\
 & &\le\,c|D(s)|t\nu((x-y)/4)\nonumber\\
 & &\le\,c|D(s)|t\frac{xy}{|x-y|^2}\nu((x-y)/(8\sqrt{2})),
\end{eqnarray}
where the last step follows from (\ref{levymeasure}) and
(\ref{asympt_infty}). Combining \pref{estdensH3}, \pref{estdensH4}
and \pref{estdensH5} after dividing by $|D(s)|$ and passing
$s\searrow0$ we obtain for $x,y\ge1$, $|x-y|\ge1$ and $|x-y|^2\ge
256t$,
\begin{equation}\label{estdensH6}p^{(0,\infty)}_t(x,y)\le
c\frac{xy}{|x-y|^{2}}\left(g_{t}\left((x-y)/(8\sqrt{2})\right)+
t\nu\left((x-y)/(8\sqrt{2})\right)\right).\end{equation}

 Next we consider $|x-y|^2\le 256t$.  By \pref{density1A1} we get for
$t\ge 1/256$ and $x,y\ge 1$,
\begin{equation}\label{estdensH7}
p^{(0,\infty)}_t(x,y)\le c t^{-1/2}\frac{xy}{t}.
\end{equation}
Since for $t>|x-y|^2/256\ge 1/256$,
$$ct^{-1/2}\le g_{t}((x-y)/(8\sqrt{2}))$$
we obtain
$$\label{estdensH14}
p^{(0,\infty)}_t(x,y)\le
c\frac{xy}{t}g_{t}((x-y)/(8\sqrt{2})),\quad t>|x-y|^2/256.
$$
 The above inequality  combined with \pref{estdensH6}, \pref{estdensH7} and Lemma \ref{transden2}  implies the
first claim of the theorem.

To prove the second conclusion of the theorem  we apply
\pref{estdensH6} for $256t<|x-y|^2$ and  \pref{estdensH7}
 for $256t\ge
|x-y|^2$ to get
\begin{eqnarray*}
& &\!\!\! \int^{\infty}_1t^{-(d-1)/2}p^{(0,\infty)}_t(x,y)dt\\
& & \le\, c\frac{xy}{|x-y|^2} \int^{
|x-y|^2/256}_{1/256}t^{-(d-1)/2}\left(g_{t}\left(\frac{x-y}{8\sqrt{2}}\right)+t\nu\left(\frac{x-y}{8\sqrt{2}}\right)\right)dt
\\&&\;+\,c\,xy\int^\infty_{|x-y|^2/256}t^{-d/2-1}dt\\
& &\le\, c\, xy\left(\frac{|x-y|^{2-d}}{|x-y|^2}+\frac{e^{-|x-y|/16}}{|x-y|^{\alpha/2+3}}(1\vee |x-y|^{5-d})+\frac{1}{|x-y|^d}\right)\\
& &\le\, c\frac{xy}{|x-y|^d}.
\end{eqnarray*}
Note that we used \pref{asympt_infty}  to estimate the density of
the L\'evy measure.
\end{proof}

\section{Green function of $\H\subset \Rd$, $d\ge 2$.}
\setcounter{equation}{0}

In this section we extend our one-dimensional estimates for a
half-line to higher dimensions. To achieve this we start with some
upper estimates of the transition densities of the killed process.

Note that by subordination we have
$$p_t(x)=Eg_{T_{\alpha}(t)}(x),\quad x\in \Rd.$$

Let  $A_t=\{\omega:\beta t <T_{\alpha}(t)(\omega)<2t\}$ be the set
defined in the proof of Corollary \ref{density1A}. Let us define
$$q_t(x)=E\left(g_{T_{\alpha}(t)}(x);{A_t^c}\right),\quad x\in \Rd.$$

In the sequel we will need a simple upper bound of $q(t,x)$. Note
that $g_t(x)\le \frac c{|x|^d},\ t>0,$ and this used for $q_t(x)$
yields
\begin{equation}\label{tail}q_t(x)\le \frac c{|x|^d}P(A_t^c) \le \frac C{|x|^d}e^{-2 t}.
\end{equation}

The next lemma will have a very important role in obtaining the
upper bound for the Green function. We introduce the following
notation. For $x\in \Rd$ we denote ${\bf x}=(x_1,\dots,x_{d-1})$
and by  ${\bf g}_t( {\bf x})$ we denote the Brownian semigroup in
$\mathbb{R}^{d-1}$.

\begin{lem}\label{density10}There is a constant $C=C(d,\alpha)$ such that
\begin{equation}\label{density01}p^\H_t(x,y)\le C{\bf
g}_{ 2t}( {\bf x}-{\bf y})p^{(0,\infty)}_t( x_d,y_d) +
q_t(x-y),\quad  x, y \in \H.\end{equation}
\end{lem}

\begin{proof}
 For $y \in \H$ and $\delta>0$ denote $V=V_y(\delta)= [y,y+\delta]=\times_{i=1}^d[y_i,y_i+\delta]=
 {\bf V}\times V_d\subset \H$. Then by independence of the
 subordinator $T_{\alpha}(t)$ and Brownian motion $B_t$  one gets
\begin{eqnarray*}
& &\!\!\! P^x(X_t\in V, \tau_\H>t )\\& &=\, P^x(X_t\in V, \tau_\H>t, A_t )+ P^x(X_t\in V, \tau_\H>t, A^c_t )\\
& &=\, E \Big[A_t;P^{\bf x}\left( {\bf B}_{T_{\alpha}(t)}\in {\bf
V}|T_{\alpha}(\cdot) \right)\times\\
& & \ \ \times\: P^{ x_d}\left(B^{(d)}d_{T_\alpha(t)}\in V_d, B^{(d)}_{T_\alpha(s)}>0; 0<s<t|T\alpha(\cdot) \right)\Big]\\
& &\ \ +\:  P^x(X_t\in V, \tau_\H>t, A^c_t )\\
& &\le\, \sup_{\beta t\le u\le 2 t} P^{\bf x}( {\bf B}_{u}\in {\bf
V}) P^{ x_d}(B^{(d)}_{T_\alpha(t)}\in V_d,
B^{(d)}_{T_\alpha(s)}>0: 0<s<t )\\& &\ \ +\: \int_{V} q_t(x-z)dz\\
& &\le\, CP^{\bf x}( {\bf B}_{2 t}\in {\bf V}) P^{
x_d}(X^{(d)}_{t}\in V_d, \tau_\H>t)+ \int_{V} q(t,x-z)dz.
\end{eqnarray*}
After dividing both sides by $|V|$ and  passing $ \delta \searrow
0$ we obtain the conclusion.

\end{proof}
Note that for any $x,y\in \H$ we can estimate ${\bf g}_{ 2t}( {\bf
x}-{\bf y})\le ct^{-(d-1)/2}$ so from \pref{density01} we deduce
that   \begin{equation}\label{density2} p^\H_t(x,y)\le
ct^{-(d-1)/2}p^{(0,\infty)}_t(x_d,y_d) + q_t(x-y),
\end{equation}
which will be well estimated with the help of Theorem
\ref{density1}.

 Lemma \ref{density10}, the estimate (\ref{tail}) and Theorem
\ref{density1} show that for the points $x,y\in\H$ away from the
boundary such that $|x-y|>2$  the Green functions for the
relativistic process and the Brownian motion are comparable. In
view of the one-dimensional case this result, proved below, is not
surprising.

\begin{thm}\label{GaussEstimate}
 For $|x-y|>2$ and  $x_d, y_d\ge 1$ we have
$$G_\H(x,y) \approx    G_\H^{gauss}(x,y).$$

\end{thm}
\begin{proof}
The lower bound follows from Lemma \ref{potential_lower}.

 We claim that the following  upper bound holds:
 \begin{equation}\label{Gauss10}
 G_\H(x,y)\le c\frac {x_d y_d}{|x - y|^d}.
 \end{equation}
By \pref{density1A1},
$$p^{(0,\infty)}_t(x_d,y_d)\le C \frac{x_d y_d} {t^{3/2}},\quad  t\ge 1,$$
which together with \pref{density01}  and (\ref{tail}) yield  the
following bound for the transition density
$$p^\H_t(x,y)\le  C {\bf g}_{2t}({\bf y}-{\bf x})\frac{x_d y_d} {t^{3/2}} + ce^{-2t}|x-y|^{-d}, \quad t\ge 1.$$
 Integrating it over $(1, \infty)$  we arrive at
\begin{equation}\label{Gauss102}\int_1^\infty p^\H_t(x,y)dt\le C \frac{x_d y_d} {|{\bf y}-{\bf x}|^d}+
\frac c{|x-y|^d}\le C_1 \frac{x_d y_d} {|{\bf y}-{\bf x}|^d}.
\end{equation}

If $|x_d - y_d|\ge 1$ we apply  \pref{density2}, (\ref{tail}) and
Theorem \ref{density1} to arrive at
\begin{eqnarray}\int_1^\infty p^\H_t(x,y)dt
&\le& C\int_1^\infty
t^{-(d-1)/2}p^{(0,\infty)}_t(x_d,y_d)dt + \frac c{|x-y|^d}\nonumber\\
&\le& C\frac {x_d y_d}{|x_d -
y_d|^d}.\label{Gauss101}\end{eqnarray}
Next note that by Lemma  \ref{transden2}  we can estimate
$$\int_0^1 p^\H_t(x,y)dt\le \int_0^1 p_t(x-y)dt\le \frac c{|x-y|^d}.$$
This combined with \pref{Gauss102} and \pref{Gauss101} implies
\pref{Gauss10}.

Now let $d\ge3$.  Since   (see \pref{0-potential}),
$$G_\H(x,y)\le  C   \frac {1}{|x - y|^{d-2}}$$  we have the following bound for
$|x-y|>2$,
$$G_\H(x,y)\le  C    \min\kl   \frac {x_d y_d}{|x - y|^d},\frac{1}{|x-y|^{d-2}}\kr\approx G^{gauss}_\H(x,y),$$
 where the last equivalence follows from (\ref{gaussGreen}). This  completes the proof in this case.

Now we finish the proof for  $d=2$. By \pref{Gauss10}, for
$\frac{x_2y_2}{|x-y|^2}\le 1$, we have
$$G_\H(x,y)\le C \frac{x_2y_2}{|x-y|^2} \approx  \ln \left(1+4\frac{x_2y_2}{|x-y|^2}\right)={2\pi}G^{gauss}_\H(x,y),$$
 where the last equality is just (\ref{gaussGreen2}).
If $\frac{x_2y_2}{|x-y|^2}>1$, using Lemmas \ref{transden2} and
\ref{density} together  with Theorem \ref{optimaltail} we obtain
\begin{eqnarray*}
G_\H(x,y)&\le&\int^{|x-y|^2}_0p_t(x-y)dt\\& &\; +\,
c\int^{\infty}_{|x-y|^2}t^{-1}P^x(\tau_\H>t/3)P^y(\tau_\H>t/3)dt\\
&\le&\int^{|x-y|^2}_0\frac{c}{|x-y|^2}dt+C\int^{x_2y_2}_{|x-y|^2}t^{-1}dt
+ Cx_2y_2 \int^\infty_{x_2y_2}t^{-2}dt\\
&\le& c +
C\ln\left(\frac{x_2y_2}{|x-y|^2}\right) \\
&\le& C\ln\left(1+4\frac{x_2y_2}{|x-y|^2}\right)=2C \pi
G^{gauss}_\H(x,y),
\end{eqnarray*}
which completes the proof for $d=2$.
\end{proof}

Now we are ready to prove the main result of this section.
\begin{thm}\label{greenhlafspaceest}
For $d\ge 3$ and $x,y\in \H$:
$$G_\H(x,y)\approx \min\left\{\frac{ (x_d\vee x_d^{\alpha/2})
(y_d \vee y_d^{\alpha/2})}{|x-y|^d}, \frac1{|x-y|^{d-2}}\right\},
\  |x-y|> 3,
$$
$$G_\H(x,y)\approx
\left[\left(\frac{ x_d\wedge y_d}{|x-y|}\right)^{\alpha/2}
\wedge1\right] \frac{1 }{|x-y|^{d-\alpha}}, \  |x-y|\le 3.$$
 For $d= 2$ and $x,y\in \H$:
 $$G_\H(x,y)\approx  \ln \left(1+ 4\frac{(x_2\vee x_2^{\alpha/2})
(y_2 \vee y_2^{\alpha/2})}{|x-y|^2}\right),\  |x-y|> 3,
$$
$$G_\H(x,y) \approx
\left[\left(\frac{ x_2\wedge y_2}{|x-y|}\right)^{\alpha/2}
\wedge1\right] \frac{1 }{|x-y|^{2-\alpha}} +\ln(1\vee( x_2\wedge
y_2)), \ |x-y|\le 3.$$
\end{thm}

\begin{proof}
 First assume $|x-y|\le 3$. In the  paper
\cite{ByczRyzMal} it was proved
$$G_\H(x,y)\approx G^1_\H(x,y), \quad d\ge 3$$
and
$$G_\H(x,y)\approx G^1_\H(x,y)+\ln(1\vee( x_2\wedge y_2)) , \quad d= 2.$$
Since $|x-y|\le 3$ by Theorem \ref{Green1est} we obtain
$$G_\H(x,y)\approx
\left[\left(\frac{x_d\wedge y_d}{|x-y|}\right)^{\alpha/2}
\wedge1\right] \frac{1 }{|x-y|^{d-\alpha}},\quad d\ge 3$$ and
$$ G_\H(x,y)\approx
\left[\left(\frac{x_2\wedge y_2}{|x-y|}\right)^{\alpha/2}
\wedge1\right] \frac{1 }{|x-y|^{2-\alpha}} +\ln(1\vee( x_2\wedge
y_2)),\quad d=2.$$ This yields the bound in the case $|x-y|\le 3$.

We introduce the following notation $\tilde{x}=(x_1,\dots,
x_{d-1}, 1\vee x_d )$, $x^*=(x_1,\dots, x_{d-1}, 0)$.
 Now  assume  $|x-y|> 3$ and observe that implies that $|x-y|\approx |\tilde{x}-\tilde{y}|> 2$.

 Then if both points are away from the
boundary ( $x_d\wedge y_d\ge 1$) we use  Lemma \ref{GaussEstimate}
to have
$$G_\H(x,y)\approx G^{gauss}_\H(x,y)= G^{gauss}_\H(\tilde{x},\tilde{y}).
$$
 Next suppose that
$x_d< 1\le y_d$.
 Let
$D(x^*)=B(x^*,\sqrt{2})\cap\H$. Then $y\notin D(x^*)$ and
$G_\H(\cdot,y)$ is a regular harmonic function on $D(x^*)$
vanishing on $\H^c$. Hence by BHP (see Lemma \ref{BHP1}) and next
by Theorem \ref{GaussEstimate}  we have
$$G_\H(x,y)\approx  x_d^{\alpha/2} G_\H(\tilde{x},y)= x_d^{\alpha/2} G_\H(\tilde{x},\tilde{y})\approx  x_d^{\alpha/2}G^{gauss}_\H(\tilde{x},\tilde{y}).$$

A similar argument applies for $x_d, y_d< 1$. Notice that $x
\notin D(y^*)$ and $y \notin D(x^*)$. Hence  $G_\H(\cdot,y)$ and
$G_\H(x,\cdot)$ are  regular harmonic function on $D(x^*)$ and
$D(y^*)$, respectively, vanishing on $\H^c$. Hence Lemma
\ref{BHP1}  and Theorem \ref{GaussEstimate} imply that
$$G_\H(x,y)\approx x_d^{\alpha/2}y_d^{\alpha/2}G_\H(\tilde{x},\tilde{y})
 \approx x_d^{\alpha/2}y_d^{\alpha/2}G^{gauss}_\H(\tilde{x},\tilde{y}).$$
Taking into account all cases we have
 $$G_\H(x,y)\approx (1\wedge x_d)^{\alpha/2}(1\wedge y_d)^{\alpha/2}G^{gauss}_\H(\tilde{x},\tilde{y}), \quad |x-y|> 3.$$
 Applying (\ref{gaussGreen}) and (\ref{gaussGreen2}) we can rewrite the above bound
 as
 $$G_\H(x,y)\approx (1\wedge x_d)^{\alpha/2}(1\wedge y_d)^{\alpha/2}\min\left\{\frac{ (x_d\vee 1) (y_d \vee 1) }{|\tilde{x}-\tilde{y}|^d}, \frac1{|\tilde{x}-\tilde{y}|^{d-2}}\right\},$$
for $ d\ge 3$, and
 $$G_\H(x,y)\approx (1\wedge x_d)^{\alpha/2}(1\wedge y_d)^{\alpha/2}
  \ln\left(1+4\frac{(x_2\vee 1) (y_2\vee 1)}{|\tilde{x}-\tilde{y}|^2}\right),$$ for
  $d=2$. Taking into account $|x-y|\approx |\tilde{x}-\tilde{y}|\ge 1$, for $|x-y|> 3$, we finally arrive at
 $$G_\H(x,y)\approx \min\left\{\frac{ (x_d\vee x_d^{\alpha/2})
(y_d \vee y_d^{\alpha/2})}{|x-y|^d}, \frac1{|x-y|^{d-2}}\right\},
\quad d\ge 3
$$and
$$G_\H(x,y)\approx  \ln\left(1+4\frac{(x_2^{\alpha/2}\vee x_2)  (y_2^{\alpha/2}\vee y_2)}{|x-y|^2}\right), \quad d=2.$$
\end{proof}

Now we compare the Green functions for half-space for the
relativistic process and for the corresponding stable process, so
we recall the formula of the Green function in the stable case
(see \cite{BGR}):
$$G^{stable}_{\H}(x,y)=C(\alpha,d)|x-y|^{\alpha-d}\int_0^{\frac{4x_dy_d}{|x-y|^2}}\frac{t^{\alpha/2-1}}{(t+1)^{d/2}}dt.$$
One can derive   sharp estimates from the above formula. Our
results from Theorem \ref{greenhlafspaceest}, Theorem \ref{green}
and Theorem \ref{Green1est} show that for the points $x,y\in\H$
such that $|x-y|\le 2$ the Green functions of the half-space $\H$
for  the relativistic process and for the corresponding stable
process are comparable if  $d\ge 3$. If $d=1$ or $d=2$ they are
also  comparable but we have to assume additionally that  the
points are near the boundary.

\begin{rem}
Suppose that $|x-y|\le 2$  then
$$G_\H(x,y)\approx \left\{%
\begin{array}{ll}
    G^{stable}_{\H}(x,y), & \hbox{$d\ge 3$;} \\
    G^{stable}_{\H}(x,y)+\ln(1\vee(x_2\wedge y_2)), & \hbox{$d=2$;} \\
    G^{stable}_{\H}(x,y)+(x\wedge y)\vee (x\wedge y)^{\alpha/2}, & \hbox{$d=1$.} \\
\end{array}%
\right.    $$ From  the estimates obtained in Theorem
\ref{greenhlafspaceest} and Theorem \ref{green} we can infer  that
$$G_\H(x,y)\ge C( G^{stable}_{\H}(x,y)+G^{gauss}_{\H}(x,y)),\quad x,y\in\H.$$
\end{rem}

\section{Green functions  for intervals}
\setcounter{equation}{0}

In this section we provide optimal estimates for Green functions
of bounded intervals. We know that for any  interval    the Green
function is comparable with the corresponding Green function of
the symmetric process. That is for the interval $(0,R)$, for $R\le
R_0$, we have
\begin{equation} C(R_0)^{-1} G_{(0,R)}^{stable}(x,y) \le G_{(0,R)}(x,y)\le C(R_0)G_{(0,R)}^{stable}(x,y), \label{intervalcomparison}\end{equation}
where $0<x,y < R.$
However, if $R_0$ grows, then the constant   $C(R_0)$ tends to
$\infty$, so the above bound is not optimal in general case. The
aim of this section is to provide optimal bounds for large
intervals. We recall known estimates for stable cases:
\begin{equation}\label{greenstable}
G_{(0,R)}^{stable}(x,y)\approx \left\{%
\begin{array}{ll}
    \min\left\{\frac{1}{|x-y|^{1-\alpha}},
    \frac{(\delta_R(x)\delta_R(y))^{\alpha/2}}{|x-y|}\right\}, & \hbox{$\alpha<1$,} \\
    \ln\left(1+\frac{(\delta_R(x)\delta_R(y))^{1/2}}{|x-y|}\right), & \hbox{$\alpha=1$,} \\
    \min\left\{(\delta_R(x)\delta_R(y))^{(\alpha-1)/2},
    \frac{(\delta_R(x)\delta_R(y))^{\alpha/2}}{|x-y|}\right\}, & \hbox{$\alpha>1$,} \\
\end{array}%
\right.
\end{equation}
where $\delta_R(x)=x\wedge(R-x)$.

 We start with the proposition showing that for points $x,y$ in the first
 half of the interval  the Green function
 of the interval and the  Green function of $(0,\infty)$ are comparable.
\begin{prop}\label{Greenintervalprop1}
Let $R\ge 4$. For $x,y\le R/2+1$ we have
$$G_{(0,R)}(x,y)\approx G_{(0,\infty)}(x,y).$$
\end{prop}
\begin{proof}
Throughout the whole proof we  assume that $x\le y\le R/2+1$.
Notice that it is enough to prove that
$$G_{(0,R)}(x,y)\ge c G_{(0,\infty)}(x,y),$$ for $x< 1$ or
$|x-y|<1$. Indeed, by Lemma \ref{potential_lower} and Remark
\ref{Greenhalflineasymp} we obtain for $x\ge1$ and $|x-y|\ge1$,
$$
G_{(0,R)}(x,y)\ge \frac2\alpha G^{gauss}_{(0,R)}(x,y)=\frac2\alpha
x(1-y/R)\ge \frac1{2\alpha} x\ge c G_{(0,\infty)}(x,y).
$$

We claim that
\begin{equation}
G_{(0,R)}(x,y)\ge CG^{1}_{(0,\infty)}(x,y), \quad
|x-y|<2.\label{lowerbound100}
\end{equation}

Applying (\ref{greenpot100}) with $D_2=(0,\infty)$ and $D_1=(0,R)$
we have
\begin{eqnarray*}
& &\!\!\! G^1_{(0,\infty)}(x,y)-G^1_{(0,R)}(x,y)\\& & =\,
E^x\left[\tau_{(0,R)}<\tau_{(0,\infty)};e^{-\tau_{(0,R)}}G^1_{(0,\infty)}(X_{\tau_{(0,R)}},y)\right]\\
&&\le \,\sup_{z\ge
R}G^1_\H(z,y)P^x(\tau_{(0,R)}<\tau_{(0,\infty)}).
\end{eqnarray*}
Next, by Theorem \ref{Green1est} and Proposition \ref{MRestimate}
we obtain
$$G^1_{(0,\infty)}(x,y)-G^1_{(0,R)}(x,y)\le C e^{-R/2}R^{-2+\alpha/2}(x^{\alpha/2}\vee x).
$$
Hence  Lemma \ref{trivialbound} yields
$$G^1_{(0,\infty)}(x,y)-G^1_{(0,R)}(x,y)\le C e^{-R/2}R^{-1+\alpha/2}G^1_{(0,\infty)}(x,y).$$
This  proves (\ref{lowerbound100}) for $R>R_0$, if $R_0>4$ is
large enough.

To handle the case $4\le R\le R_0$ we apply
(\ref{intervalcomparison}) together with
  \pref{greenstable} and Theorem \ref{Green1est}  to obtain
$$G_{(0,R)}(x,y)\ge c(R_0)G^{stable}_{(0,R)}(x,y)\ge c G^1_{(0,\infty)}(x,y),$$
which ends the proof of (\ref{lowerbound100}).

That is, by Theorem \ref{green}, for $|x-y|<2,\ x\ge 1$, we get
\begin{eqnarray}\label{grintl1}G_{(0,R)}(x,y)&\ge&
c(G^1_{(0,R)}(x,y)+G^{gauss}_{(0,R)}(x,y))\nonumber\\&\ge&
c(G^1_{(0,\infty)}(x,y)+x)\ge c G_{(0,\infty)}(x,y).\end{eqnarray}
Next, for $x< 1$ and $y\ge 2$, by BHP (Lemma \ref{BHP1}),
\begin{eqnarray}\label{grintl2}G_{(0,R)}(x,y)&\ge& c G_{(0,R)}(1,y)
x^{\alpha/2}\ge c x^{\alpha/2} G_{(0,\infty)}(1,y)\nonumber\\
&\ge& c G_{(0,\infty)}(x,y) .\end{eqnarray} Combining
\pref{grintl1} and \pref{grintl2} give us
$$G_{(0,R)}(x,y)\approx G_{(0,\infty)}(x,y),$$
for $x,y\le R/2+1$.
\end{proof}

\begin{prop}\label{Greenintervalprop2}Let $R\ge 4$. Suppose that $1\le x\le y\le R-1$, and
$|x-y|\ge 1$ then
$$G_{(0,R)}(x,y)\approx  G^{gauss}_{(0,R)}(x,y)=x(R-y)/R .$$
\end{prop}
\begin{proof}
 Due to Lemma \ref{potential_lower} we only need to  prove upper bound. Suppose that $1\le x\le
 y\le R-1$ and $|x-y|\ge 1$.

At first, let additionally $y\le 3/4 R$, then by Remark
\ref{Greenhalflineasymp},
    \begin{equation}\label{grintp1}
    G_{(0,R)}(x,y)\le G_{(0,\infty)}(x,y)\le c x\le
    4cx(R-y)/R.
    \end{equation}
By symmetry and the above inequality we have
\begin{eqnarray}\label{grintp2}G_{(0,R)}(x,y)&=&G_{(-R,0)}(-x,-y)=G_{(0,R)}(R-x,R-y)\nonumber\\
&\le& 4c(R-y)x/R,\end{eqnarray}
 for $R/4\le x\le R/2\le y\le R-1$ and $|x-y|>1$.

Hence it remains to consider the case $1\le x\le R/4$ and $3R/4\le
y\le R-1$. Denote $\eta=\tau_{(0,R/2)}$. Since
$G_{(0,R)}(\cdot,y)$ is regular harmonic on $(0,R/2)$, so by
Proposition  \ref{Greenintervalprop1} and Remark
\ref{Greenhalflineasymp},
\begin{eqnarray}
& &\!\!\!G_{(0,R)}(x,y)\nonumber\\&
&=\,E^xG_{(0,R)}(X_{\eta},y)\le
E^x\left[X_{\eta}>R/2;G_{(0,\infty)}(R-X_{\eta},R-y)\right]\nonumber\\
& &\le\,cE^x\left[ X_{\eta}>R/2,|X_{\eta}-y|<1;G^1_{(0,\infty)}(R-X_{\eta},R-y)\right]\nonumber\\
& & \; + \,
cE^x\left[X_{\eta}>R/2;\left((R-X_{\eta})\vee(R-X_{\eta})^{\alpha/2}\right)\wedge
(R-y)\right]\nonumber\\
& & \le \, c(R-y)P^x(\eta<\tau_{(0,\infty)})\nonumber\\& &\;+\,cE^x\left[|X_{\eta}-y|<1;G^1_{(0,\infty)}(R-X_{\eta},R-y)\right]\nonumber\\
& &\le\, c(R-y)x/R \nonumber\\& &\;+\,
cE^x\left[|X_{\eta}-y|<1;G^1_{(0,\infty)}(R-X_{\eta},R-y)\right],\label{grintp3}
\end{eqnarray}
where the last inequality is a consequence of Proposition
\ref{MRestimate}. Moreover, by Lemma \ref{poissonKernel} we obtain
\begin{eqnarray}
& &
\!\!\!E^x\left[|X_{\eta}-y|<1;G^1_{(0,\infty)}(R-X_{\eta},R-y)\right]\nonumber\\&
& = \,
\int^{y+1}_{y-1}P_{(0,R/2)}(x,z)G^1_{(0,\infty)}(R-z,R-y)dz\nonumber\\
& &\le\,
c E^x\eta\int^{y+1}_{y-1}e^{-(z-R/2)}G^1_{(0,\infty)}(R-z,R-y)dz\nonumber\\
& &\le\, c \frac{R}{2} x e^{-(y-1-R/2)}\int^\infty_0
G^1_{(0,\infty)}(R-z,R-y)dz \nonumber\\& &\le\, c
x(R-y)/R,\label{grintp4}
\end{eqnarray}
because $y\ge 3/4R$.

Combining \pref{grintp1}, \pref{grintp2}, \pref{grintp3} and
\pref{grintp4} we obtain
$$G_{(0,R)}(x,y)\le c x(R-y)/R=cG^{gauss}_{(0,R)}(x,y).$$
\end{proof}

Now we can prove the main result of this section.
\begin{thm}
Let $R\ge 4$ and $x\le y$ then we have for $|x-y|\le 1$,
$$
G_{(0,R)}(x,y)\approx \min\{G_{(0,\infty)}(x,y),
G_{(0,\infty)}(R-x,R-y) \} $$ and for $|x-y|> 1$
$$
G_{(0,R)}(x,y)\approx \frac {(x^{\alpha/2}\vee
x)((R-y)^{\alpha/2}\vee (R-y))}R. $$
\end{thm}

\begin{proof}  Observe that by symmetry
\begin{equation}\label{symmetry}G_{(0,R)}(x,y)=G_{(0,R)}(R-x,R-y).\end{equation}
The case $|x-y|\le 1$ follows immediately from Proposition
\ref{Greenintervalprop1},
 Theorem \ref{green} and (\ref{symmetry}).

 For $0<x\le y <R$,  $|x-y|> 1$ we define $\tilde{x}= x\vee 1$ and $\tilde{y}= y \wedge (R-1)$. Then we can repeat the arguments used in the proof of Theorem \ref{greenhlafspaceest} to arrive at
   \begin{eqnarray*}G_{(0,R)}(x,y)&\approx& (1\wedge x)^{\alpha/2}(1\wedge (R-y))^{\alpha/2}G^{gauss}_{(0,R)}(\tilde{x},\tilde{y})\\
  &=& (1\wedge x)^{\alpha/2}(1\wedge (R-y))^{\alpha/2} \frac{\tilde{x}(R-\tilde{y})}R\\
  &=& \frac {(x^{\alpha/2}\vee
x)((R-y)^{\alpha/2}\vee (R-y))}R , \quad |x-y|> 1.\end{eqnarray*}
  This completes the proof.
\end{proof}


\end{document}